\documentclass[11pt,a4paper]{article}
\usepackage{amsfonts}
\usepackage{amsmath,amssymb}
\usepackage{mathrsfs}
\usepackage{hyperref}
\usepackage{graphicx}
\usepackage{amsthm}
\usepackage{color}

\newtheorem{theorem}{Theorem}[section]

\newtheorem{lemma}[theorem]{Lemma}

\newtheorem{definition}[theorem]{Definition}
\newtheorem{remark}{Remark}

\numberwithin{equation}{section}\allowdisplaybreaks
\begin{document}

\title{\large\bf  The propagation of regularity and dispersive  blow-up phenomenon to higher-order generalized KdV equations}
\author{\normalsize Minjie  Shan$^{a,*}$\\
\footnotesize
\it $^a$ College of Science, Minzu University of China, Beijing 100081, P.R. China. \\  \\ \footnotesize
\it E-mail:  smj@muc.edu.cn   \\
}
\date{} \maketitle
\thispagestyle{empty}

\begin{abstract}
Some special properties of smoothness and  singularity concerning to the initial value problem associated with higher-order generalized KdV equations are investigated. On one hand, we show the propagation of regularity phenomena. More precisely,
the regularity of initial data on the right-hand side of the real line is propagated to the left-hand side with infinite speed under the higher-order KdV flow. On the other hand, we show that the dispersive blow-up phenomenon will occur by constructing a class of smoothing initial data  such that global solutions with the given initial data keep smooth at positive generic irrational times, while  global solutions display singularity  at each time-space positive rational point. The blow-up phenomenon is exclusively caused by the linear part of solutions due to the focusing of short or long waves. 
\end{abstract}
%\linenumbers
\section{Introduction}
This paper is concerned with the initial value problem (IVP) associated to higher-order $k$-generalized KdV equations
\begin{equation}
	\left\{
	\begin{aligned}
		&\partial_{t}u +\partial_{x}^{2j+1} u +u^k\partial^j_{x} u = 0, \quad  j, k\in\mathbb{N}^+,\\
		&u(0,x)=u_0(x),\ \ \ x\in \mathbb{R}, \ t\in\mathbb{R} \label{hKdV} \\
	\end{aligned}
	\right.
\end{equation}
which is a particular case of the class of IVPs
\begin{equation}
	\left\{
	\begin{aligned}
		&\partial_{t}u +\partial_{x}^{2j+1} u +Q(u,\partial_{x} u, \cdots, \partial_{x}^{2j} u) = 0, \quad  j\in\mathbb{N}^+,\\
		&u(0,x)=u_0(x),\ \ \ x\in \mathbb{R}, \ t\in\mathbb{R}  \label{OhKdV} \\
	\end{aligned}
	\right.
\end{equation}
where
$Q: \mathbb{R}^{2j+1}\to \mathbb{R}$
 is a polynomial without constant or linear terms.   Lax \cite{Lax1965} first introduced \eqref{OhKdV} to generalize the KdV hierarchy. \eqref{OhKdV} is also a common higher-order models arising in water waves problems, elastic media with microstructure and in other physical problems \cite{KO92}.

 Kenig, Ponce and Vega  \cite{KPV94Hi1}  proved that  \eqref{OhKdV} is well-posed in weighted Sobolev spaces for small initial data by taking advantage of local smoothing effects associated to the unitary group of the linear equation. Later, utilizing  several sharp estimates for solutions of the associated linear problem and a change of dependent variable, the smallness assumption on the initial data was removed in \cite{KPV94Hi2}.  By using weighted Besov spaces, Pilod \cite{Pilod08Hi}  refined well-posedness  results for \eqref{OhKdV} with a special class of  nonlinearity and small initial data.

It was Kato who first studied well-posedness for the KdV equation (with $j=k=1$ in \eqref{hKdV}) in weighted Sobolev spaces 
$$
Z_{s,r}=H^s(\mathbb{R})\cap L^2\big(|x|^{2r}dx\big)
$$
where $s,r\in \mathbb{R}$. Kato \cite{Kato83} showed that persistent properties hold for solutions to the KdV equation for any $r\in \mathbb{N}^+$ and $s\geq 2r$. The notion of persistence properties is if the initial data locate in $Z_{s,r}$, then the associated IVP is
locally or globally well-posed (which means that solutions keep in $Z_{s,r}$ for any time $t$). In \cite{Kato83},  Kato used the following commutative property of operators
$$\Gamma=x-3t\partial_x^2, \hspace{5mm} \mathcal{L}=\partial_t+\partial_x^3, \hspace{5mm} [\Gamma, \mathcal{L}]=0$$
which deduces that
	\begin{align}
		xU(t)v_0= U(t)(xv_0)+3tU(t)(\partial^2_x v_0)\nonumber  
	\end{align}
where $U(t)$ is the unitary operator semigroup for the linear KdV equation. From the identity  above, one can easily see that the regularity of solutions to the KdV equation is twice
the decay rate of the solution. Kato's result was improved in \cite{Nahas12,FLP15} to $Z_{s,r}$ with $r>0$, $s\geq 2r$ and $s>\max\{s_k, 0\}$ where $s_k$ is the critical indicator of well-posedness for KdV in Sobolev spaces (for further details see \cite{KPV1993} and references therein).  The hypothesis $s\geq 2r$ is necessary \cite{ILP13}. In other word, if $u_0\in Z_{s,r}$ with $2r>s$, then the solution  $u(t)$ stays only in $ Z_{s,s}$ at any time $t\neq 0$ which means that the extra decay $2r-s$ is not preserved by the solution flow.

Why weighted Sobolev spaces is used?  In fact, to complete the local smoothing estimate, a maximal (in time) function estimate is needed in $L^1_x$. It was observed in \cite{KPV1993}  that the $L^1_x$-maximal function estimate fails without weight. 

As mentioned above, weighted Sobolev space is an useful tool to study regularity and decay properties of solutions to the IVPs for dispersive equations. Next, we are concerned with  a special type of regularity properties for dispersive equations. 

In \cite{ILP15}, the propagation of regularity phenomena for $k$-generalized KdV equations are described. To be specific, if initial data $u_0\in  H^{3/4+}$ and $u_0\in H^{l}((b, \infty))$ for some $l\in \mathbb{Z}^+$ (which is later extended to $l>3/4$ in \cite{KLPV18}) and $b\in\mathbb{R}$, then corresponding  solutions $u(t,x)$ are in $H^{l}((\beta, \infty))$ for any $\beta\in\mathbb{R}$ and any $t\in (0, T)$ where $T$ is the maximum existence interval.  This result indicates that  the propagation of regularity in the right hand side of the data  moves to its left with infinite speed  as time evolves. Moreover, by using weighted Sobolev spaces, it was also showed in \cite{ILP15} that corresponding solutions to $k$-generalized KdV equations possess  some persistence properties and regularity effects for positive times if the initial data $u_0\in  H^{3/4+}$  have polynomial decay in the positive real line. Subsequently, analogical properties are established for the Benjamin-Ono equation with negative dispersion \cite{ILP16a}, the dispersive generalized Benjamin-Ono  equation \cite{Mendez20a}, the fractional KdV equation  \cite{Mendez20},  the Benjamin equation \cite{GuoQin18}, the fifth-order dispersive equation  \cite{SeSm15}, the Kadomtsev-Petviashvili equation \cite{ILP16b}, the Zakharov-Kuznetsov(ZK) equation  \cite{LP18ZK, Mendez24} and the intermediate long-wave equation  \cite{MPS19}.

Relative to regularity, singularity of solutions  has aroused intense interest. Next, let us recall the dispersive blow-up phenomenon for dispersive models. Dispersive singularity was first raised for solutions to the linear KdV equation by Benjamin, Bona and Mahony \cite{1972_benjamin}. From a physics perspective, it is a type of focusing phenomenon which describes propagating waves with different speed might present strange singularities by gathering somewhere. Mathematically, the regularity for solutions to dispersive equations with smooth initial value is destroyed at some points in time-space. Dispersive blow-up phenomena widely exist in various dispersive models, such as KdV, gKdV equations \cite{1993_Bona_Dispersive_gKdV}, Schr\"odinger equations \cite{2010_Bona_dispersive_Schroinger, 2016_Hong_dispersive_NLS}. By using the smoothing effect properties, Linares and Scialom \cite{1993_LinaScia} showed dispersive blow-up for the nonlinear generalized KdV equation. Dispersive blow-up for the KdV equation was shown by Linares, Ponce and  Smith \cite{2017_Linares} via taking advantage of fractional weighted spaces. Similar results were obtained for the ZK equation in two dimensional case \cite{2020_Linares_blowupforZK} and in three dimensional case \cite{BS2024}, and for the Schr\"odinger–KdV system \cite{2019_LinaPala}. The main idea to show dispersive blow-up  is that the Duhamel term associated to solutions is smoother than the linear evolution component. To achieve this, the smoothing effects and 
weighted Sobolev spaces are combined elaborately to gain more regularity.

The purpose of this article is threefold. Firstly, we  show that persistent properties hold for solutions to the IVP \eqref{hKdV} in weighted Sobolev spaces.  Secondly, we investigate the propagation of regularity and decay of solutions. Lastly, it is  proved that the dispersive blow-up solutions exist for higher-order generalized KdV equations.

Now we state the main results. The first one is about persistence properties of solutions. This well-posedness result in weighted Sobolev space will be used to establish dispersive blow-up for higher-order generalized KdV equations later.

\begin{theorem} \label{lwp}
	Let $s\geq j+1/2$, $r\in(0, 1)$ and $s>2jr$. Assume that $u_0\in Z_{s,r}$, then there exists a positive time $T=T(\|u_0\|_{Z_{s,r}})>0$, such that \eqref{hKdV} with $k=1$ has a unique solution
	$$ u(t,x)\in C\big([0,T]; Z_{s,r})$$
which depends continuously upon  $u_0$. Moreover, we have
		\begin{align}
\Big\|J^{s-\frac{2j+1}{4}-}u\Big\|_{L_x^{2}L^{\infty}_{T}}+\big\|J^{s} \partial_x^j u\big\|_{L_x^{\infty}L^2_{T}}+ \Big\|J^{j+1/2}D_x^{\frac{2j-1}{4}}u\Big\|_{L_T^{2}L^{\infty}_{x}}< \infty.\label{lwp001} 
		\end{align}
 and 
		\begin{align}
\Big\|J^{s}D^{\frac{2j-1}{8}}_{x}  u\Big\|_{L^{8}_{T}L^{4}_{x}}+\Big\|J^{s}D^{\frac{2j-1}{6}}_{x}  u\Big\|_{L^{6}_{T}L^{6}_{x}}< \infty.\label{lwp002} 
		\end{align}
\end{theorem}

The second result  is about  the propagation of regularity in the right hand
side of the data for positive times. It indicates that the regularity of solutions travels to the left with infinite speed as time progresses. 
\begin{theorem} \label{prop-reg}
Let $m\in \mathbb{N}$, $m\geq j+1$ and $x_0\in \mathbb{R}$. Assume that $u_0\in H^{j+1/2}(\mathbb{R})$ and 
		\begin{align}
\big\|\partial_x^{m}u_0(x)\big\|^2_{L^{2}((x_0, \infty))}=\int_{x_0}^{\infty}|\partial_x^{m}u_0(x)|^2dx< \infty, \label{prop-reg0a} 
		\end{align}
 then the solution $u$ to  \eqref{hKdV} on $[0, T]$ satisfies that for any $v> 0$ and  $\varepsilon> 0$
		\begin{align}
\sup_{0\leq t \leq T} \int_{x_0+\varepsilon-vt}^{\infty}|\partial_x^{\ell}u|^2(t,x)dx<c, \label{prop-reg01} 
		\end{align}
for $\ell=0,1,\cdots,m$ with $c=c\big(m; \|u_0\|_{ H^{j+1/2} };\big\|\partial_x^{m}u_0\big\|_{L^{2}((x_0, \infty))};v; \varepsilon; T\big)$. In particular, for all $t\in(0,T]$, we have 
 $$u(t,\cdot)\in H^m((x_0,\infty)).$$
 Moreover, for any $v\geq 0$,  $\varepsilon> 0$ and  $R> 0$
		\begin{align}
\int_0^T  \int_{x_0+\varepsilon-vt}^{x_0+R-vt}|\partial_x^{m+j}u|^2(t,x)dxdt<c  \label{prop-reg02} 
		\end{align}
 with $c=c\big(m; \|u_0\|_{ H^{j+1/2} };\big\|\partial_x^{m}u_0\big\|_{L^{2}((x_0, \infty))};v; \varepsilon; R; T\big)$.
\end{theorem}
\begin{remark}
From the local smoothing effect, see \eqref{lwp001}, one only can get
 		\begin{align}
\int_0^T  \int_{x_0+\varepsilon-vt}^{x_0+R-vt}|D_x^{2j+1/2}u|^2(t,x)dxdt\leq(R-\varepsilon)\big\|D_x^{2j+1/2}u\big\|_{L^{\infty}_xL^{2}_T}<\infty  \label{prop-reg03} 
		\end{align}
which tells us that  \eqref{prop-reg02} is an improvement result of \eqref{prop-reg03}.
\end{remark}

\begin{remark}
The persistence properties and regularity effects for solutions to $k$-generalized
KdV equations  with polynomial decay initial data were studied in \cite{ILP15} at length. We would like to make a statement that an analogous result holds for \eqref{hKdV} without trying to prove it here.  Specifically, assume that $u_0\in H^{j+1/2}(\mathbb{R})$ and 
		\begin{align}
\big\|x^{\frac{m}{2j}}u_0(x)\big\|^2_{L^{2}((0, \infty))}=\int_{0}^{\infty}|x^{\frac{m}{2j}}u_0(x)|^2dx< \infty \nonumber 
		\end{align}
for some $m\in \mathbb{N}$ and $m\geq j+1$,  then the solution $u$ to  \eqref{hKdV} on $[0, T]$ satisfies that 
		\begin{align}
\sup_{0\leq t \leq T} \int_{0}^{\infty}|x^{\frac{m}{2j}}u(t, x)|^2dx<c \nonumber 
		\end{align}
with $c=c\big(m; \|u_0\|_{ H^{j+1/2} };\big\|x^{\frac{m}{2j}}u_0\big\|_{L^{2}((0, \infty))};T\big)$. 

 Moreover, for any $\varepsilon, \delta> 0$, $v\geq 0$,  $\ell_1, \ell_2\in \mathbb{N}$, $\ell_1\geq j$ and $\ell_1+ \ell_2\leq m$,
		\begin{align}
\sup_{\delta\leq t \leq T} \int_{\varepsilon-vt}^{\infty}\big|x_{+}^{\frac{\ell_1}{2j} }\partial_x^{\ell_2}u\big|^2dx+\int_{\delta}^T  \int_{\varepsilon-vt}^{\infty}\big|x_{+}^{\frac{\ell_1-j}{2j} }\partial_x^{\ell_2+j}u\big|^2dxdt<c  \nonumber
		\end{align}
 with $c=c\big(m; \|u_0\|_{ H^{j+1/2} }; \big\|x^{\frac{m}{2j}}u_0\big\|_{L^{2}((0, \infty))};T;\delta; \varepsilon; v\big)$.
\end{remark}

The third result is concerned with the existence of dispersive blow-up solution to higher-order generalized KdV equations \eqref{hKdV}. 

In order to state the dispersive blow-up theorem accurately, we need the following definition for generic irrational number. 

In \cite{DGG17}, Deng, Germain and Guth introduced the definition of genericity. 
\begin{definition}[see Definition 1.1 in \cite{DGG17}]\label{DGG17def}
Let $m,\ell$ be positive integers. Assume that   $\beta_{\ell} \in I$
for all $1\leq \ell\leq m$, where $I$ is a fixed interval of $\mathbb{R}$. We will call a property generic in $(\beta_1, \cdots, \beta_m)$ if it is true for all $(\beta_1, \cdots, \beta_m)$
outside of a null set (set with measure zero) of $I\times   \cdots \times I$.
\end{definition}

It is well-known (see \cite{Cassels57}) that,  generically in $(\beta_1, \cdots, \beta_m)$, then
\begin{equation}
|k_1+k_2\beta_2+ \cdots+k_m\beta_m|\gtrsim \frac{1}{(|k_1|+ \cdots+|k_m|)^{m-1} \log(|k_1|+ \cdots+|k_m|)^{2m}}.\label{generic0}
\end{equation}
By using Definition \ref{DGG17def}, we explain what is generic irrational number.
\begin{definition} \label{gene-irrat}
We will call a real number $r_0$ generic irrational number if $r_0$ is a  irrational number and $(r,r_0)$ is outside of a null set of $ (\mathbb{R}\setminus\mathbb{Q})^2$ for all $r\in \mathbb{R}\setminus\mathbb{Q}$. And we denote $r_0\in \mathbb{R}\setminus\mathbb{Q}^{*}$ if $r_0$ is a generic irrational number.
\end{definition}

One can  immediately get from \eqref{generic0} that
\begin{equation}
\left|\frac{k_1}{k_2}-r_0\right|\gtrsim \frac{1}{(|k_1|+|k_2|)^{3}} \label{generic2}
\end{equation}
for all generic irrational number  $r_0\in \mathbb{R}\setminus\mathbb{Q}^{*}$.  Moreover, it is easy to see that $\mathbb{Q}\subset \mathbb{Q}^*$ and $\mathbb{Q}^*\setminus\mathbb{Q}$ is a set with measure zero.

\begin{theorem} \label{dispersiveblow-upsolution}
Assume that  $s\in [j+1, j+3/2)$, $r\in(0, 1)$ and $s>2jr$, then there exists $u_0 \in Z_{s,r}\cap C^{\infty}(\mathbb{R})$ such that the solution  $u(t)$ of \eqref{hKdV} with $k=1$ is global in time satisfying
	\begin{equation}
		\left\{
		\begin{aligned}
			&u(t)\in C^{j+1}(\mathbb{R}), \hspace{43mm} t>0, \ t\in  \mathbb{R}\setminus\mathbb{Q}^{*}, \quad \\
			&u(t)\in C^{j+1}(\mathbb{R}\setminus \mathbb{Q}^+)\setminus C^{j+1}(\mathbb{R}), \hspace{16.2mm}  t>0, \ t\in\mathbb{Q}. \nonumber
		\end{aligned}
		\right.
	\end{equation}
	Moreover, the Duhamel term
	\begin{equation}
		\begin{aligned}
			z_1(t)=\int_0^t W(t-t')(u\partial^j_x u)(t')dt'
			\nonumber
		\end{aligned}
	\end{equation}
	is in $C^{j+1}(\mathbb{R})$ for all $t>0$.
\end{theorem}

\begin{remark}
This theorem  shows that the dispersive blow-up phenomenon is exclusively caused by singularities from the linear part of \eqref{hKdV}.
\end{remark}

The construction of linear dispersive blow-up solutions allows us to extend the result described above to  solutions of \eqref{hKdV} with $k\geq2$. In these cases, weighted Sobolev space is not indispensable. 

\begin{theorem} \label{dispblupkg2}
Let  $k\geq2$, $s=j+3/2-$ and $0<r<1$. Then there exists $u_0 \in H^{s}\cap C^{\infty}(\mathbb{R})$ with $\|u_0\|_{H^s}\ll1$ such that the solution  $u(t)$ of \eqref{hKdV} is global in time satisfying
	$$u(t)\in C(\mathbb{R};H^s(\mathbb{R}))\cap X_T^k$$
where $X_T^k$  is the work space defined via Strichartz estimates, the smoothing effect estimates and maximal functions estimates. Moreover, we have
	\begin{equation}
		\left\{
		\begin{aligned}
			&u(t)\in C^{j+1}(\mathbb{R}), \hspace{43mm} t>0, \ t\in  \mathbb{R}\setminus\mathbb{Q}^{*}, \quad \\
			&u(t)\in C^{j+1}(\mathbb{R}\setminus\mathbb{Q}^+)\setminus C^{j+1}(\mathbb{R}), \hspace{16.2mm} t>0, \ t\in \mathbb{Q}. \nonumber
		\end{aligned}
		\right.
	\end{equation}
\end{theorem}

\begin{remark}
The proof for this theorem is very similar to that of Theorem  \ref{dispersiveblow-upsolution}, hence we omit the details.
\end{remark}

From the consequences mentioned above, we know that the Duhamel term possesses higher regularity.
\begin{theorem} \label{ghKdVDuhamelkg2}
Let $k\geq2$, $s\geq j+1$ and $s\in \mathbb{N}$. Assume that $u(t)\in C([-T, T];H^s(\mathbb{R}))$ is the solution to \eqref{hKdV} with initial data $u_0 \in H^s(\mathbb{R})$. Denote
$$z_k(t)=\int_0^t W(t-t')(u^k\partial^j_x u)(t')dt',$$
then  we have
	$$z_k(t)\in C([-T, T];H^{s+j}(\mathbb{R})).$$
\end{theorem}

Our last result is a supplement to Theorem  \ref{prop-reg}. To be precise, we will construct initial data $u_0\in H^{j+1}(\mathbb{R})\cap W^{r, p}(\mathbb{R})$ for some $r$ and $p$, such that the singularities of solutions do not propagate in any direction.
\begin{theorem} \label{DoNotProg}
Let $k\geq2$ and $t^*\neq 0$. 
\begin{itemize}
\item[1.] There exist $u_0\in H^{j+1}(\mathbb{R})\cap W^{j+1, p}(\mathbb{R})$, $p>2$, such that the corresponding solution to \eqref{hKdV} $u\in C(\mathbb{R};  H^{j+1}(\mathbb{R}))$ is global in time and satisfies:
 $$u(t^*)\notin W^{j+1,p}(\mathbb{R}) \hspace{8mm} \text{for every} \hspace{3mm} p>2.$$
\item[2.] There exist $r>j+1$, $p>2$ and an initial datum $u_0\in H^{j+1}(\mathbb{R})\cap W^{r, p}(\mathbb{R})$,  such that the corresponding solution to \eqref{hKdV} $u\in C(\mathbb{R};  H^{j+1}(\mathbb{R}))$ is global in time and satisfies:
 $$u(t_0)\notin W^{r,p}(\mathbb{R}_{+}) \hspace{8mm} \text{and}\hspace{8mm}  u(-t_0)\notin W^{r,p}(\mathbb{R}_{+})$$
for some $t_0>0$, where $\mathbb{R}_{+}:=\{x\in \mathbb{R}: x\geq 0\}$. The same result holds for $\mathbb{R}_{-}$.
\end{itemize}
\end{theorem}

\textbf{Notation.} We give the notation that will be used throughout this paper. For $A, B \geq 0$ fixed,  $A\lesssim B$ means that $A\leq C \cdot B$ for an absolute constant $C>0$.  $A\gg B$ means that  $A>C \cdot B$ for a very large positive constant $C$.  We write $c+\equiv c+\epsilon$ and $ c-\equiv c-\epsilon$ for some  $0<\epsilon\ll 1$.

 We denote spatial variables by $x$ and its dual Fourier variable  by $\xi$. Given a function $u$, we denote $\mathscr{F} u $ or $\widehat{u}$ its Fourier transform and denote $\mathscr{F}^{-1} u $  its Fourier inverse transform.   The unitary group associated to the linear higher-order KdV equation is given by
$$ W(t)=e^{-t\partial_x^{2j+1}}=\mathscr{F}^{-1}e^{it(-1)^{j+1}\xi^{2j+1}}\mathscr{F}.$$  Then, the solution to \eqref{hKdV} can be written as
\begin{equation}\label{IZK}
	u(t)=W(t)u_0+\int^t_0 W(t-t')(u^k\partial^j_x u)dt'.\nonumber
\end{equation}

Let $1\leq p,q\leq\infty$. We define 
$$\|f\|_{L^p_xL^q_{T}}=\left(\int_{\mathbb{R}}\Big(\int^T_{-T}|f(t,x)|^qdt \ \Big)^{p/q}dx\right)^{1/p}$$
with the usual modifications if either $p=\infty$ or $q=\infty$ .  If $T=\infty$ we shall use the notation $\|f\|_{L^p_xL^q_{t}}$. Similar definitions and considerations may be made interchanging the variables $x$ and $t$.

For $s>0$, we also define $D_x^sf$ and $J^sf$ as
	\begin{align}
D^s_x f =\mathscr{F}^{-1}|\xi|^{s}\hat{f}(\xi), \hspace{10mm} J^s f =\mathscr{F}^{-1}(1+\xi^2)^{s/2}\hat{f}(\xi).	\label{DeriFouri}
\end{align}

\textbf{Organization of the paper.} In Section 2, we recall some
estimates that will be used in the proofs that follow. Section 3 begins with the computation of free solution in fractional weighted Sobolev space and then proceeds to the proof  the  persistence  property. The treatment of propagation of one-sided regularity for solutions to \eqref{hKdV} comprises Section 4 where Theorem \ref{prop-reg} is proved.  We show Theorem \ref{dispersiveblow-upsolution} in Section 5 which consists of two parts. Subsection 5.1 is devoted to the construction of smooth initial data such that the solution of the corresponding linear equation develops singularities at all positive rational times.  Subsection 5.2 is devoted to the smoothing of Duhamel term. We also prove  Theorem \ref{ghKdVDuhamelkg2} and Theorem \ref{DoNotProg} at the end of Subsection 5.2.

%%%%%%%%%%%%%%%%%%%%%%%%%%%%%%%%%%%%%%%%%%%%%%%%%%%%%%%%%%%%%%%%%%%%%%%
\section{Preliminaries}\label{section:linear estimate}

We recall some important estimates in this section, such as  Strichartz estimates,  local smoothing estimates, maximal function estimates, interpolation inequality and commutator estimates.

%%%strichartz
Let us first give the  dispersive decay estimate for the linear operator of higher-order KdV equations. 
\begin{lemma}[see Lemma 2.7 in \cite{KPV91}]\label{lem:dispersive decay}
	Let $j\in \mathbb{N}^+$ and $\beta\in \mathbb{R}$.  Denote
	\begin{equation*}
		I_{t}(x)=\int_{\mathbb{R}}|\xi|^{\frac{2j-1}{2}+i\beta} e^{it(-1)^{j+1}\xi^{2j+1}+i x\xi} d \xi,
	\end{equation*}
then we have
	\begin{align}
		\left\|I_{t}(x)\right\|_{L^\infty} \lesssim (1+|\beta|)|t|^{-1/ 2}.\label{hKdVdispEs}
	\end{align}
\end{lemma}

This inequality implies the following Strichartz estimates by using a Stein-Tomas type argument.

\begin{lemma}[Strichartz estimates, see Theorem 2.1 in \cite{KPV91}]\label{linear estimate:strichartz}
	Let $0\leq\theta\leq1$, $1\leq p ,q ,\tilde{p},\tilde{q}\leq\infty$  and $\frac{1}{p}+\frac{1}{p'}=\frac{1}{q}+\frac{1}{q'}=\frac{1}{\tilde{p}}+\frac{1}{\tilde{p}'}=\frac{1}{\tilde{q}}+\frac{1}{\tilde{q}'}=1$. Then
	\begin{align}
		\left\|D_x^{\frac{\theta(2j-1)}{4}}W(t)u_0\right\|_{L_t^qL^p_{x}} &\lesssim
		\|u_0\|_{L^2_{x}},\label{ineq:strichartz1} \\
	\left\|\int D_x^{\frac{\theta(2j-1)}{4}}W(-t')g(t', \cdot)dt'\right\|_{L^2_{x}} &\lesssim
		\|g\|_{L_t^{q'}L^{p'}_{x}}, \label{ineq:strichartz2}\\
	\left\|\int D_x^{\frac{\theta(2j-1)}{2}}W(t-t')g(t', \cdot)dt'\right\|_{L_t^qL^p_{x}} &\lesssim
		\|g\|_{L_t^{\tilde{q}'}L^{\tilde{p}'}_{x}}, \label{ineq:strichartz3}
	\end{align}
where $(q,p)=(\frac{4}{\theta}, \frac{2}{1-\theta})$ and $\frac{4}{\tilde{q}}+\frac{2}{\tilde{p}}=1$. In particular, by taking $\theta=1$, we have
	\begin{equation}
		\Big\|D^{\frac{2j-1}{4}}_x W(t)u_0\Big\|_{L^2_T L^\infty_{x}}\lesssim T^{1/4} \|u_0\|_{L^2_{x}}.\label{linear estimate: strichartz cor1}
	\end{equation}
\end{lemma}

%%%% smoothing effect
Next is Kato's smoothing effects which greatly helps us deal with the higher-order derivative nonlinear term.
\begin{lemma}
	[Local smoothing estimates, see Theorem 2.1 and Corollary 2.2 in \cite{KPV94Hi2}]\label{Kato's smoothing effect }
	Let $j\in \mathbb{N}^+$ and $T>0$. We have
\begin{align}
\big\|\partial_x^j W(t)u_0 \big\|_{L^2_{t}} &=c
		\|u_0 \|_{L^{2}_{x}},\label{linear estimate:Kato-1} \\
	\left\|\partial_x^{j}\int_0^t W(-t')g( t', \cdot)\mathrm{d}t'\right\|_{L_x^{2}} &\lesssim
		\|g\|_{L^1_xL^{2}_{t}},\label{linear estimate:Kato-2a}\\
	\left\|\partial_x^{2j}\int_0^t W(t-t')g( t', \cdot)\mathrm{d}t'\right\|_{L_x^{\infty}L^2_{t}} &\lesssim
		\|g\|_{L^1_xL^{2}_{t}},\label{linear estimate:Kato-2}\\
	 \left\|\partial_x^{j+l}\int_0^t W(t-t')g( t', \cdot)\mathrm{d}t'\right\|_{L_x^{\infty}L^2_{T}} &\lesssim
		T^{(j-l)/2j}\|g\|_{L^p_xL^{2}_{T}}.\label{linear estimate:Kato-3}
	\end{align}
with $l=0,1,\cdots,j$, and $p=2j/(j+l)$.
\end{lemma}

%%%%%%%%%%%%%%%% maximal
To complement the above estimates we need to bound the $L^2$-norm
of the maximal function $\sup_{[0,T]} |W(t) u_0(x)|$.
\begin{lemma}[Maximal function estimate,  see Theorem 2.3 in \cite{KPV94Hi2}]
For $s> (2j+1)/4$, we have
	\begin{equation}
		\|W(t) u_0\|_{L_{x}^{2} L_{T}^{\infty}} \lesssim (1+T)^{\frac{3}{4}+}\|u_0\|_{H^{s}}. \label{linear estimate:maximal}
	\end{equation}
\end{lemma}

The interpolation inequality is given in the next lemma.

\begin{lemma}[see Lemma 2.7 in \cite{2020_Linares_blowupforZK}]\label{lem:interpolaion}
	Assume that $a, b>0$, $p\in(1,\infty)$ and  $\theta \in (0,1)$. If
	 $J^a f\in L^p(\mathbb{R}^n)$  and $\left< x\right>^b f\in L^p(\mathbb{R}^n)$, then
	\begin{equation}
		\|\left< x\right>^{(1-\theta)b}J^{\theta a}f\|_{L^p}\lesssim \|\left< x\right>^b f\|_{L^p}^{1-\theta}\|J^af\|_{L^p}^{\theta}. \label{lem:interpolaion1}
	\end{equation}
The same holds for homogeneous derivatives $D^a$ in place of $J^a$. Moreover, for $p=2$,
	\begin{equation}
		\|J^{\theta a}\big(\left< x\right>^{(1-\theta)b}f\big)\|_{L^2}\lesssim \|\left< x\right>^b f\|_{L^2}^{1-\theta}\|J^af\|_{L^2}^{\theta}. \label{lem:interpolaion2}
	\end{equation}
\end{lemma}

%%%%%%%%%%%%Leibniz rules

The following classical Kato-Ponce commutator estimate \cite{KP1988} plays an important role in
the well-posedness theory of Navier-Stokes and Euler equations and KdV equation in Sobolev
spaces.
\begin{lemma} \label{Kato-Ponce ineq}
	Let $s>0$ and $p\in (1,\infty)$. Then
	\begin{equation}
		\|J^s(f g)-f J^s g\|_{L^p(\mathbb{R})}\lesssim 		\|J^sf\|_{L^p(\mathbb{R})}\|g\|_{L^{\infty}(\mathbb{R})}+\|\partial_x f\|_{L^\infty(\mathbb{R})}\|J^{s-1}g\|_{L^p(\mathbb{R})}. \label{Kato-Ponce ineq01}
	\end{equation}
\end{lemma}

There are many other generalisations of Kato-Ponce commutator estimates  (cf. \cite{BO14, GO14, MS13} and the references therein). The following two kinds of fractional Leibniz rules  will be used to show well-posedness and nonlinear smoothing for \eqref{hKdV}.

\begin{lemma}[see Theorem 1 in \cite{KPV1993}]\label{lem:Leibniz}
	Let $s\in(0,1)$ and $p\in (1,\infty)$. Then
	\begin{equation}
		\|D^s(f g)-f D^s g-g D^s f\|_{L^p(\mathbb{R})}\lesssim \|g\|_{L^\infty(\mathbb{R})}\|D^sf\|_{L^p(\mathbb{R})}. \label{lem:Leibniz-1}
	\end{equation}
Further more, we have
	\begin{equation}
		\|D^s(f g)\|_{L^p(\mathbb{R})}\lesssim \|f D^s g\|_{L^p(\mathbb{R})} + \|g\|_{L^\infty(\mathbb{R})}\|D^s f\|_{L^p(\mathbb{R})}. \label{lem:Leibniz-12}
	\end{equation}
\end{lemma}

\begin{lemma}[see Theorem 1.2 in \cite{Li19}]\label{lem:Leibnizpoly}
	Let $s>0$ and $1<p, p_1,  p_2<\infty$ with $1/p=1/p_1+1/p_2$. Then for any $s_1, s_2 \geq 0$ with $s_1+ s_2=s$, and any
$f, g \in \mathcal{S}(\mathbb{R}^n)$, the following inequality holds:
	\begin{equation}
		\Big\|D^s(f g)-\sum_{|\alpha|\leq s_1}\frac{1}{\alpha!}\partial_x^{\alpha}f D^{s,\alpha} g-\sum_{|\beta|\leq s_2}\frac{1}{\beta!}\partial_x^{\beta}g D^{s,\beta} f\Big\|_{L^p}\lesssim \|D^{s_1}f\|_{L^{p_1}} \|D^{s_2}g\|_{L^{p_2}}  \label{lem:Leibnizpoly1}
	\end{equation}
where the operator $D^{s,\alpha}$ is defined via Fourier transform  as $$\widehat{D^{s,\alpha}g}(\xi)=i^{-|\alpha|}\partial_{\xi}^{\alpha}|\xi|^s .$$
\end{lemma}

We also need the  weighted Kato-Ponce inequality.

\begin{lemma}[see Theorem 1.1 in \cite{2016_Cruz-Uribe_KatoPonce}] \label{Kato-Ponce-weight}
	Let $1<p, q<\infty$, $\frac{1}{2}<r<\infty$ such that $\frac{1}{r}=\frac{1}{p}+\frac{1}{q}$. If $v \in A_{p}$, $w \in A_{q}$, and $s>\max \left\{0, n\left(\frac{1}{r}-1\right)\right\}$ or $s$ is non-negative even, then for any $f, g \in \mathcal{S}\left(\mathbb{R}^{n}\right)$, we have
	\begin{align}
		\left\|D^{s}(f g)-f D^{s}g\right\|_{L^{r}\left(v^{\frac{r}{p}} w^{\frac{r}{q}}\right)}
		\lesssim\left\|D^{s} f\right\|_{L^{p}(v)}\|g\|_{L^{q}(w)}+\|\nabla f\|_{L^{p}(v)}\left\|D^{s-1} g\right\|_{L^{q}(w)},	\label{Kato-Ponce-weight1}
		\\
		\left\|J^{s}(f g)-f J^{s}g\right\|_{L^{r}\left(v^{\frac{r}{p}} w^{\frac{r}{q}}\right)}
		\lesssim\left\|J^{s} f\right\|_{L^{p}(v)}\|g\|_{L^{q}(w)}+\|\nabla f\|_{L^{p}(v)}\left\|J^{s-1} g\right\|_{L^{q}(w)},\label{Kato-Ponce-weight2}
	\end{align}
	where the constants depend on $p, q, s,[v]_{A_{p}}$  and $[w]_{A_{q}}$.
\end{lemma}

%%%%%%%%%%%%%%%%%%%%%%%%%%%%%%%%%%%%%%%%%%%%%%%%%%%%%%%%%%%%%%%%%%%%%%%%%%%
\section{Persistence properties}

In this section we treat one of the main topics: persistence properties for solutions to higher-order generalized KdV equations. Local well-posedness  for \eqref{hKdV} in weighted Sobolev space $Z_{s,(r_1,r_2)}$ is established. For persistence properties in other dispersive models, we refer to \cite{FP11p, FLP12p, FLP13p, NP09p, 2016_Bustamante, 2016_Fonseca, BJM24} and reference therein.

Noting that $$e^{it(-1)^{j+1}\xi^{2j+1}}\partial_{\xi}\widehat{u_0}=\partial_\xi(e^{it(-1)^{j+1}\xi^{2j+1}}\widehat{u_0})+t(-1)^{j+1}(2j+1)\xi^{2j}e^{it(-1)^{j+1}\xi^{2j+1}}\widehat{u_0},$$
we derive 
\begin{equation*}
	\begin{aligned}
		\quad W(t)\left(x u_{0}\right)
		&=i\mathscr{F}^{-1} \partial_{\xi}\left(e^{it(-1)^{j+1}\xi^{2j+1}} \widehat{u}_{0}\right)+ t(-1)^{j}(2j+1)\mathscr{F}^{-1}\xi^{2j} e^{it(-1)^{j+1}\xi^{2j+1}} \widehat{u_{0}} \\
		&= xW(t)u_0+(2j+1)t W(t)\partial_{x}^{2j} u_{0}.
	\end{aligned}
\end{equation*}

The above identity suggests that the regularity of solutions to higher-order generalized KdV equations is $2j$ times larger than the decay rate.

Next, let us recall the definition of Stein derivation (see \cite{Stein1961} or \cite{Stein_HarmonicAnalysis}). For $\alpha\in (0,2)$ and $x\in\mathbb{R}^n$, define
	\begin{align}
\mathcal{D}_{\alpha}f(x) = \lim_{\varepsilon\to 0}\frac{1}{c_{\alpha}}\int_{|y|\geq\varepsilon}\frac{f(x+y)-f(x)}{|y|^{n+\alpha}} dy, \label{Stein-deri}
	\end{align}
where $c_{\alpha}=\frac{\pi^{n/2}\Gamma(-\alpha/2)}{2^{\alpha}\Gamma((n+2)/2)}$. It was remarked in \cite{Stein1961} that  \eqref{Stein-deri} is consistent with	\eqref{DeriFouri}, i.e. 
$$\mathcal{D}_{\alpha}f(x)=D^{\alpha}f(x)=\mathscr{F}^{-1}|\xi|^{\alpha}\widehat{f}(\xi)$$
for $f \in \mathscr{S}(\mathbb{R}^n)$. 

Denote
$$W^{\alpha,p}=(1-\Delta)^{-\alpha/2}L^p(\mathbb{R}^n).$$
Stein   gave the following equivalent characterization of the Sobolev space $W^{\alpha,p}$ in \cite{Stein1961}. 

$f\in W^{\alpha,p}(\mathbb{R}^n)$ if and only if $f\in L^p(\mathbb{R}^n)$ and $\mathcal{D}_{\alpha}f\in L^p(\mathbb{R}^n)$, where $\alpha\in(0,2)$ and $p\in(1, \infty)$.
Moreover,
$$\|f\|_{W^{\alpha,p}}:=\|(1-\Delta)^{\alpha/2}f\|_{L^p}\simeq \|f\|_{L^p}+\|D^{\alpha}f\|_{L^p}\simeq \|f\|_{L^p}+\|\mathcal{D}_{\alpha}f\|_{L^p}.$$

Denote
	\begin{align}
\Lambda_{\alpha}(F(\cdot, y))(x) = \lim_{\varepsilon\to 0}\frac{1}{c_{\alpha}}\int_{|y|\geq\varepsilon}\frac{F(x, y)}{|y|^{n+\alpha}} dy, \nonumber
	\end{align}
then we have the following Leibnitz's rule for $\mathcal{D}_{\alpha}$
	\begin{align}
\mathcal{D}_{\alpha}(fg)(x) = g(x)\mathcal{D}_{\alpha}f(x)+\Lambda_{\alpha}\big((g(\cdot+ y)-g(\cdot))f(\cdot+ y)\big)(x). \label{Stein-deriLeiRule}
	\end{align}
Taking $g(x)=e^{i\phi(x)}$ in  \eqref{Stein-deriLeiRule} implies 
	\begin{align}
\mathcal{D}_{\alpha}(e^{i\phi(x)}f)(x) = e^{i\phi(x)}\mathcal{D}_{\alpha}f(x)+e^{i\phi(x)}\Lambda_{\alpha}\big((e^{i(\phi(\cdot+y)-\phi(\cdot))}-1)f(\cdot+ y)\big)(x). \label{Stein-deriLeiRule2}
	\end{align}

\begin{lemma}\label{lem:weightZK}
	Let $r\in (0, 1) $ and  $s\geqslant 2jr$. Assume that  $u_0\in Z_{s,r} (\mathbb{R})$, then for all $ t\in \mathbb{R}$ and almost every $x\in\mathbb{R}$, it holds that	
	\begin{equation}
		|x|^{r}W(t)u_0=W(t)(|x|^{r}u_0)+W(t)\left(\left\{\Phi_{\xi,t,r}(\widehat{u_0})\right\}^{\vee}\right),\label{eq:weight_x}
	\end{equation}
	with
\begin{align}
		\|\Phi_{\xi,t,r}(\widehat{u_0})\|_2&\lesssim (1+|t|)\|u_0\|_{H^s(\mathbb{R})}.\label{eq:error_x}
	\end{align}
	Moreover, if $0<\beta< r$,  $D^{\beta}(|x|^{r}u_0)\in L^2(\mathbb{R})$ and  $u_0\in H^{s+\beta}(\mathbb{R})$, then one has
	\begin{equation}
		D^{\beta}\left(|x|^{r}W(t)u_0\right)=W(t)(D^{\beta}|x|^{r}u_0)+W(t)\left(D^{\beta}\left\{\Phi_{\xi,t,r}(\widehat{u_0})\right\}^{\vee}\right), \label{eq:weight_derive}
	\end{equation}
	with
\begin{align}
		\|D^{\beta}\left\{\Phi_{\xi,t,r}(\widehat{u_0})\right\}^{\vee}\|_2&\lesssim (1+|t|)\|u_0\|_{H^{s+\beta}(\mathbb{R})}. \label{eq:error_derive} 
	\end{align}
\end{lemma}

\vspace{1mm}
\begin{proof} According to Stein derivation and \eqref{Stein-deriLeiRule2}, we get
	\begin{align}
		& |x|^{r_1}W(t)u_0-W(t)(|x|^{r}u_0) \notag \\
		=&\mathscr{F}^{-1}\left(D_\xi^{r}(e^{it(-1)^{j+1}\xi^{2j+1}} \widehat{u_0})-e^{it(-1)^{j+1}\xi^{2j+1}} D_\xi^{r_1} \widehat{u_0}\right)\notag  \\
		=&\mathscr{F}^{-1}e^{it(-1)^{j+1}\xi^{2j+1}} \Lambda_{r}\big(  (e^{it(-1)^{j+1}((\xi+y)^{2j+1}-\xi^{2j+1})}-1) \widehat{u_0}(\cdot+y) \big)(\xi)
\notag  \\
=&W(t)\left(\left\{\Phi_{\xi,t,r}(\widehat{u_0})\right\}^{\vee}\right),\label{WPhi1}
	\end{align}
where
$$\Phi_{\xi, t, r}\left(\widehat{u_0}\right)=\Lambda_{r}\big(  (e^{it(-1)^{j+1}((\xi+y)^{2j+1}-\xi^{2j+1})}-1) \widehat{u_0}(\cdot+y) \big)(\xi).$$

By using the same argument provided in \cite{FLP15} where well-posedness for generalized KdV equations in fractional weighted Sobolev spaces was studied, one derives
	\begin{align}
\|\Phi_{\xi,t,r}(\widehat{u_0})\|_{L^p}	\lesssim (1+|t|)\big(\|\widehat{u_0}\|_{L^p}	+\big\||\xi|^{2jr}\widehat{u_0}\big\|_{L^p}\big),\label{WPhi2}
	\end{align}
and 
	\begin{align}
\big\||\xi|^{\beta}\Phi_{\xi,t,r}(\widehat{u_0})\big\|_{L^p}	\lesssim (1+|t|)\big(\|\widehat{u_0}\|_{L^p}	+\big\||\xi|^{\beta+2jr}\widehat{u_0}\big\|_{L^p}\big),\label{WPhi3}
	\end{align}
Then, \eqref{eq:weight_x} and \eqref{eq:error_x} are direct results of \eqref{WPhi1} and  \eqref{WPhi2}.

For $\beta \in\left(0, r\right)$, it is easy to verify that
	\begin{align}
		&\quad D_{x}^{\beta}\left(|x|^{r} W(t) u_{0}\right)- W(t)\left(D_{x}^{\beta}|x|^{r} u_{0}\right) \notag \\
		&=\mathscr{F}^{-1}\Big(|\xi|^{\beta}D_{\xi}^{r}e^{it(-1)^{j+1}\xi^{2j+1}}\widehat{u_0}-e^{it(-1)^{j+1}\xi^{2j+1}}|\xi|^{\beta}D_{\xi}^{r}\widehat{u_0}\Big)  \notag\\
		&=\mathscr{F}^{-1}|\xi|^{\beta}e^{it\omega(\xi,\eta)}\Phi_{\xi,t,r}  \notag\\
		&=W(t)\left(D^{\beta}\left\{\Phi_{\xi,t,r}(\widehat{u_0})\right\}^{\vee}\right). \label{DWPhi4}
	\end{align}

From \eqref{WPhi3} and \eqref{DWPhi4}, we immediately obtain \eqref{eq:weight_derive} and \eqref{eq:error_derive}. So, we finish the proof. 
\end{proof}

Let us turn to the well-posedness in weighted Sobolev spaces. We define the work space as
$$X_T=\Big\{u\in C([0,T];Z_{s,(r_1,r_2)}):\|u\|_{X_T} <\infty \Big\}$$
where
	\begin{align}
		\|u\|_{X_T}= &
		\|u\|_{L^{\infty}_TH^s_{x}}+\big\||x|^{r}u\big\|_{L^{\infty}_TL^2_{x}} +\Big\|J^{s-\frac{2j+1}{4}-}u\Big\|_{L_x^{2}L^{\infty}_{T}}+\big\|J^{s} \partial_x^j u\big\|_{L_x^{\infty}L^2_{T}}\notag \\
&+ \Big\|J^{j+1/2}D_x^{\frac{2j-1}{4}}u\Big\|_{L_T^{2}L^{\infty}_{x}}
+\Big\|J^{s}D^{\frac{2j-1}{8}}_x u\Big\|_{L_T^{8}L^{4}_{x}}+\Big\|J^{s}D^{\frac{2j-1}{6}}_{x}  u\Big\|_{L^{6}_{xT}}.\label{workspaceXT}
	\end{align}

\begin{lemma}\label{lwplem}
Let $ s\geq j+1/2$ and $0<T\leq 1$. Assume that $u\in X_T$ where  $ X_T$  is defined via the norm given in \eqref{workspaceXT}. Denote
$$ z_1(t)=\int^t_0 W(t-t')(u\partial_x^j u)(t')dt',$$
then  we have
\begin{align}
\|z_1(t)\|_{L^\infty_T H^s_{x}}\lesssim  T^{1/2}\|u\|_{X_T}^2 .\label{lwplem0}
	\end{align}
\end{lemma}

\begin{proof}  First of all, it is easy to see that
\begin{align}
\|z_1(t)\|_{L^\infty_T L^2_{x}} \leq \int_0^T  \|u\partial_x^ju\|_{L^2_{x}} dt 
\lesssim   T^{1/2} \|u\|_{L_T^{\infty}L^2_{x}}\|\partial_x^ju\|_{L_T^{2}L^{\infty}_{x}}
\lesssim  T^{1/2}\|u\|_{X_T}^2. \nonumber
	\end{align}
	
Secondly, by using Leibniz's rule for fractional derivatives \eqref{lem:Leibnizpoly1} with $s_1=s-1/4$ and $s_2=1/4$, H\"older's inequality and Sobolev's inequality, we get
\begin{align}
&\|D^{s}_{x}z_1(t)\|_{L^\infty_T L^2_{x}}
\leq \int_0^T  \|D^{s}_{x}(u\partial_x^ju)\|_{L^2_{x}} dt \notag \\
\lesssim \ & T^{\frac{1}{2}}  \Big( \sum_{\ell\leq s-\frac{1}{4}}\big\|\partial^{\ell}_{x}uD^{s,\ell}_x\partial^{j}_{x}u\big\|_{L^2_{xT}}+\big\|\partial^{j}_{x}uD^{s}_xu\big\|_{L^2_{xT}} \Big)+ T^{\frac{7}{8}} \big\|D^{s-\frac{1}{4}}_x u\big\|_{L^{\infty}_{T}L^4_{x}} \big\|D^{\frac{1}{4}}_x \partial^{j}_{x}u \big\|_{L^8_{T}L^4_{x}}\notag \\
\lesssim \ & T^{\frac{1}{2}}  \Big( \sum_{\ell< s-\frac{2j+1}{4}}\big\|\partial^{\ell}_{x}u\big\|_{L_x^{2}L^{\infty}_{T}}\big\|D^{s,\ell}_x\partial^{j}_{x}u \big\|_{L_x^{\infty}L^2_{T}}
+\sum_{ s-\frac{2j+1}{4}\leq \ell}\big\|\partial^{\ell}_{x}u\big\|_{L^{\infty}_{T}L_x^{2}}\big\|D^{s,\ell}_x\partial^{j}_{x}u \big\|_{L^2_{T}L_x^{\infty}}\Big)\notag \\
& +T^{\frac{1}{2}}\|D^{s}_{x}u\|_{L_T^{\infty}L^2_{x}}\|\partial_x^{j}u\|_{L_T^{2}L^{\infty}_{x}}+T^{\frac{7}{8}} \big\|J^{\frac{1}{4}}D^{s-\frac{1}{4}}_x u\big\|_{L^{\infty}_{T}L^2_{x}} \big\|D^{\frac{1}{4}}_x \partial^{j}_{x}u\big\|_{L^8_{T}L^4_{x}} \notag \\
\lesssim \ & T^{\frac{1}{2}}  \Big( \Big\|J^{s-\frac{2j+1}{4}-}u\Big\|_{L_x^{2}L^{\infty}_{T}}\big\|J^{s} \partial_x^j u\big\|_{L_x^{\infty}L^2_{T}}
+\|u\|_{L^{\infty}_{T}H_x^{s}}\Big\|J^{j+1/2}D_x^{\frac{2j-1}{4}}u\Big\|_{L_T^{2}L^{\infty}_{x}}\Big)\notag \\
& +T^{\frac{1}{2}}\|u\|_{L^{\infty}_{T}H_x^{s}}\Big\|J^{j+1/2}D_x^{\frac{2j-1}{4}}u\Big\|_{L_T^{2}L^{\infty}_{x}}+T^{\frac{7}{8}} \|u\|_{L^{\infty}_{T}H_x^{s}} \Big\|J^{s}D^{\frac{2j-1}{8}}_x u\Big\|_{L_T^{8}L^{4}_{x}} \notag \\
\lesssim_T \ & \|u\|_{X_T}^2. \nonumber
	\end{align}

Therefore, the proof is completed. 
\end{proof}

Now we consider the local well-posedness for  \eqref{hKdV} with $k=1$ in  weighted Sobolev spaces.

\textbf{Proof of Theorem \ref{lwp}. }Note that 
$$u= \chi_T(t)W(t)u_0-\chi_T(t)\int^t_0W(t-t')(u \partial_x^j u)(t')dt':=\mathscr{T} u ,$$
where $\chi_T$ is the usual smooth cut-off function. We shall estimate $\mathscr{T} u$ by using each norm in $X_T$.

\vspace{1mm}
\noindent {\bf{(i)  Estimate for $ \|\mathscr{T} u\|_{H^s_{x}}$.}}

\vspace{1.5mm}
By  \eqref{lwplem0}, one easily gets 
\begin{align}
	\|\mathscr{T} u\|_{L^{\infty}_{T}H^s_{x}}\leq  \|u_0\|_{H^s(\mathbb{R})}+\|z(t)\|_{L^{\infty}_{T}H^s_{x}}\lesssim  \|u_0\|_{H^s(\mathbb{R})}+T^{1/2}\|u\|^2_{X^T}.\label{Hs0}
\end{align}

\noindent {\bf{(ii)  Estimate for $\big\||x|^{r}\mathscr{T} u\big\|_{L^2_{x}}$.}}

\vspace{1.5mm}
It follows from  Minkovski's inequality, \eqref{eq:weight_x}, \eqref{eq:error_x} and Lemma \ref{lwplem} that
\begin{align}
	&\quad \big\||x|^{r}\mathscr{T} u\big\|_{L^2_{x}}
	\lesssim  \big\||x|^{r}W(t)u_0\big\|_{L^2_{x}}+\int_0^T\big\||x|^{r}W(t-t')(u\partial_x^j u)\big\|_{L^2_{x}}dt'\notag\\
	&\lesssim  \big\||x|^{r}u_0\big\|_{L^2_{x}}+(1+T)\|u_0\|_{H^s}+\int_0^T\big\||x|^{r}u \partial_x^j u \big\|_{L^2_{x}}dt +(1+T)\int_0^T\|u\partial_x^j u\|_{H^s} dt  ,\notag\\
&\lesssim  \big\||x|^{r}u_0\big\|_{L^2_{x}}+(1+T)\|u_0\|_{H^s}+T^{\frac{1}{2}}\big\||x|^{r}u\big\|_{L_T^{\infty}L^2_{x}}\|\partial_x^j u\|_{L_T^{2}L^{\infty}_{x}}+(1+T)T^{\frac{1}{2}}\|u\|_{X_T}^2 ,\notag\\
&\lesssim (1+T)\|u_0\|_{Z_{s, r}}+(1+T)T^{\frac{1}{2}}\|u\|_{X_T}^2. \label{xHs1}
\end{align}

\noindent {\bf{(iii)  Estimate for $\Big\|J^{s-\frac{2j+1}{4}-}\mathscr{T} u\Big\|_{L^{2}_xL^{\infty}_{T}}$.}}

\vspace{1.5mm}
By applying Minkovski's inequality, the maximal function estimate  \eqref{linear estimate:maximal} and Lemma \ref{lwplem}, we deduce
\begin{align}
	\Big\|J^{s-\frac{2j+1}{4}-}\mathscr{T} u\Big\|_{L^{2}_xL^{\infty}_{T}}&\lesssim   (1+T)^{\frac{3}{4}+}\Big(\|u_0\|_{H^{s}(\mathbb{R})}+\int_0^T\|u \partial^j_x u\|_{H^{s}(\mathbb{R})}d t \Big)\notag\\
	&\lesssim  (1+T)^{\frac{3}{4}+}\Big(\|u_0\|_{H^s(\mathbb{R})}+T^{1/2}\|u\|_{X_T}^2\Big).\label{xHs2}
\end{align}

\noindent {\bf{(iv)  Estimate for $\big\|J^s \partial^j_x  \mathscr{T} u \big\|_{L^{\infty}_xL^{2}_{T}}$.}}

\vspace{1.5mm}
By using Kato smoothing effect \eqref{linear estimate:Kato-1} and Lemma \ref{lwplem}, we  obtain
\begin{align}
	\big\|J^s \partial^j_x  \mathscr{T} u \big\|_{L^{\infty}_xL^{2}_{T}}&\lesssim  \|u_0\|_{H^{s}}+\int_0^T\|u\partial^j_x u\|_{H^{s}}dt \notag\\
	&\lesssim \|u_0\|_{H^s(\mathbb{R})}+T^{1/2}\|u\|_{X_T}^2.\label{xHs3}
\end{align}

\noindent {\bf{(v)  Estimate for $\Big\|J^{j+1/2}D^{\frac{2j-1}{4}}_{x} \mathscr{T} u\Big\|_{L_T^{2}L^{\infty}_{x}}$.}}

\vspace{1.5mm}
From  \eqref{linear estimate: strichartz cor1}  and Lemma \ref{lwplem}, one has
\begin{align}
		 \Big\|J^{j+1/2}D^{\frac{2j-1}{4}}_{x} \mathscr{T} u\Big\|_{L_T^{2}L^{\infty}_{x}}
		&\lesssim T^{1/4} \Big(\|u_0\|_{H^{s}}+\int_0^T\|u\partial^j_x u\|_{H^{s}} dt \Big) \notag\\
		&\lesssim T^{1/4} \Big(\|u_0\|_{H^s(\mathbb{R})}+T^{1/2}\| u\|^2_{X_T}\Big).\label{xHs4}
\end{align}

\noindent {\bf{(vi)  Estimate for $\Big\|J^{s}D^{\frac{2j-1}{8}}_{x} \mathscr{T} u\Big\|_{L_T^{8}L^{4}_{x}}$ and $\Big\|J^{s}D^{\frac{2j-1}{6}}_{x} \mathscr{T} u\Big\|_{L^{6}_{xT}}$ .}}

\vspace{1.5mm}
Applying Strichartz's estimate \eqref{ineq:strichartz1} with $\theta=1/2$, $p=4$ and $q=8$ yields
\begin{align}
		 \Big\|J^{s}D^{\frac{2j-1}{8}}_{x} \mathscr{T} u\Big\|_{L_T^{8}L^{4}_{x}}
		&\lesssim  \|u_0\|_{H^{s}}+\int_0^T\|u\partial^j_x u\|_{H^{s}} dt  \notag\\
		&\lesssim  \|u_0\|_{H^s(\mathbb{R})}+T^{1/2}\| u\|^2_{X_T}.\label{xHs5}
\end{align}
Similarly, one can get
\begin{align}
		\Big\|J^{s}D^{\frac{2j-1}{6}}_{x} \mathscr{T} u\Big\|_{L^{6}_{xT}}
		\lesssim  \|u_0\|_{H^s(\mathbb{R})}+T^{1/2}\| u\|^2_{X_T}.\label{xHs6}
\end{align}

 Then, \eqref{Hs0}-\eqref{xHs6} help imply that
\begin{align}
	\|\mathscr{T} u\|_{X_T}&< C_1(1+T)^{3/4+} \left(\|u_0\|_{Z_{s,r}}+T^{1/2}\|u\|_{X_T}^2\right).\label{lwpcontra1}
\end{align}
where $C_1$ is a positive constant.  A similar argument leads to the estimate
\begin{align}
	\|\mathscr{T} u-\mathscr{T} v\|_{X_T}&< C_1(1+T)^{3/4+} T^{1/2}(\|u\|_{X_T}+\|v\|_{X_T})	\|u- v\|_{X_T}.\label{lwpcontra2}
\end{align}

Hence, it follows from \eqref{lwpcontra1} and \eqref{lwpcontra2}  that  $\mathscr{T}$ is a contraction mapping on 
$$ B_r=\left\{u\in X_T\ \big| \hspace{2mm} \|u\|_{X_T} <r\right\}$$  
with $r=4C_1\|u_0\|_{Z_{s,r}}$ and $T=\min\left\{1,  (4C_1r)^{-2}\right\}$. Consequently, there exists a unique solution $u$ to \eqref{hKdV} and 
\begin{align}
	\| u\|_{X_T} \leqslant 4C_1\|u_0\|_{Z_{s,r}}.\nonumber
\end{align}
We finish the proof of this theorem.\hspace{86mm}$\square$

\begin{remark}\label{LWPH^s}
From the proof of Theorem \ref{lwp}, we see that higher-order generalized KdV equations  \eqref{hKdV} are local well-posedness in $H^s(\mathbb{R})$ for $s\geq j+1/2$. Moreover, \eqref{lwp001}  and \eqref{lwp002}  also hold true.
\end{remark}

%%%%%%%%%%%%%%%%%%%%%%%%%%%%%%%%%%%%%%%%%%%%%%%%%%%%%%%%%%%%%%%%%%%%%%
 \section{Propagation of regularity}\label{PropRegu}

In this section, we focus on  discussing one-sided propagation of regularity for solutions to \eqref{hKdV}. We show Theorem \ref{prop-reg} by making use of the algebraic structure of higher-order generalized KdV equations  and local well-posedness results in $H^s(\mathbb{R})$, mainly \eqref{lwp001}  and \eqref{lwp002} (see Remark \ref{LWPH^s}). It is worth mentioning that we do not utilize weighted Sobolev spaces.

 Before stating our proof for Theorem \ref{prop-reg}, we  list the following properties concerning cutoff functions that will be used later.

\begin{lemma}\label{PropRegu-lem}
Let $\varepsilon>0$, $b\geq 5\varepsilon$  and $c_{\ell}>0$. Then there exists a real function $\chi_{\varepsilon,b}\in C^{\infty}(\mathbb{R})$  
\begin{equation}
	\chi_{\varepsilon,b}(x)=\left\{
	\begin{aligned}
		&0,\ \ \ x\leq \varepsilon ; \\
		&1,  \ \ \ x\geq b ,  \nonumber
	\end{aligned}
	\right.
\end{equation}
satisfying 
\begin{align}
&\hspace{10mm} \text{supp}\chi_{\varepsilon,b}\subset [\varepsilon, \infty), \hspace{10mm} \text{supp}\chi'_{\varepsilon,b}\subset [\varepsilon, b],\label{Chi0a1}   \\
&\chi_{\varepsilon,b}(x)\geq  \chi_{\varepsilon,b}(3\varepsilon) \geq \frac{\varepsilon}{2(b-3\varepsilon)},  \hspace{18mm} \text{for}  \hspace{5mm} x\geq 3\varepsilon,   \label{Chi0a2}    \\
&c_{\ell}\left|\chi^{(\ell)}_{\varepsilon,b}(x)\right|\leq  \chi'_{\varepsilon/3,b+\varepsilon}(x) \leq \frac{1}{b-3\varepsilon},  \hspace{11mm} \text{for}  \hspace{5mm} \ell \geq 1,  \label{Chi0a3}    \\
&\chi'_{\varepsilon,b}(x) \lesssim  \chi'_{\varepsilon/3,b+\varepsilon}(x)  \chi_{\varepsilon/3,b+\varepsilon}(x),  
\hspace{5mm} \chi'_{\varepsilon,b}(x) \lesssim  \chi_{\varepsilon/5,\varepsilon}(x). \label{Chi0a4}    
\end{align}
\end{lemma}
\begin{proof}
See \cite{ILP15}.
\end{proof}

\textbf{Proof of Theorem \ref{prop-reg}. } We use the induction argument. For simplicity, we only consider $k=1$.  And we may assume that $x_0=0$ without loss of generality.

\vspace{1mm}
\noindent {\bf{ Case 1 $m=j+1$.}} 

\vspace{1mm}
Let us first show \eqref{prop-reg01} for $\ell=0$. Multiplying the equation \eqref{hKdV} by $u(t,x)\chi_{\varepsilon,b}(x+vt)$ gives  
\begin{align}
&\frac{1}{2}\frac{d}{dt} \int u^2\chi_{\varepsilon,b}(x+vt)dx- v\int u^2\chi'_{\varepsilon,b}(x+vt)dx\notag \\
&+\int\partial_x^{2j+1}u  u \chi _{\varepsilon,b}(x+vt)dx+\int u \partial_x^{j}u u \chi _{\varepsilon,b}(x+vt)dx=0 \label{PropRegu0a0ell0a}
	\end{align}

Note that
\begin{align}
v\int_0^T\int u^2\chi'_{\varepsilon,b}(x+vt)dxdt\lesssim T \|u\|^2_{L^{\infty}_TH^{j+1/2}_x}<c. \label{PropRegu0a0ell0a1}
	\end{align}
A direct calculation deduces  
\begin{align}
\partial_x^{2j+1}u u =\frac{1}{2}\sum_{\ell=0}^{j}c_{\ell}\partial_x^{2\ell+1}\left((\partial_x^{j-\ell}u)^2\right)
\label{PRdd2}
	\end{align}
where coefficient $c_0,\cdots,c_{j}$   are determined by the following linear equation system
\begin{equation}
	\left\{
	\begin{aligned}
		&\sum_{m\leq \ell\leq j} c_{\ell}\binom{2\ell+1}{\ell-m}=0,\ \ \ m=0,1,\cdots,j-1; \\
		&\hspace{3mm} c_{j}=1.  \label{PropRegu-coeffi}
	\end{aligned}
	\right.
\end{equation}
It is worth mentioning that $c_0 \neq0$. After integration by parts, one sees  
\begin{align}
\int\partial_x^{2j+1}u  u \chi _{\varepsilon,b}(x+vt)dx
=&\frac{1}{2}\sum_{\ell=0}^{j}c_{\ell}\int \partial_x^{2\ell+1}\left((\partial_x^{j-\ell}u)^2\right)\chi _{\varepsilon,b}(x+vt)dx\notag \\
=&-\frac{1}{2}\sum_{\ell=0}^{j}c_{\ell}\int (\partial_x^{j-\ell}u)^2\chi^{(2\ell+1)} _{\varepsilon,b}(x+vt)dx\nonumber
	\end{align}
From local well-posedness result, for $0\leq \ell\leq j$, we get
\begin{align}
\int_0^T \int \left|(\partial_x^{j-\ell}u)^2\chi^{(2\ell+1)} _{\varepsilon,b}(x+vt)\right|dx dt\lesssim T \|u\|^2_{L^{\infty}_TH^{j+1/2}_x}<c. \label{PropRegu0a0ell0a2}
	\end{align}
Moreover,
\begin{align}
\left|\int u \partial_x^{j}u u \chi _{\varepsilon,b}(x+vt)dx\right| 
\leq& \|\partial_x^{j}u\|_{L^{\infty}_x}\int u^2 \chi _{\varepsilon,b}(x+vt)dx\notag \\
\leq &\|u\|_{L^{\infty}_TH^{j+1/2}_x}\int u^2 \chi _{\varepsilon,b}(x+vt)dx. \label{PropRegu0a0ell0a3}
	\end{align}

Inserting \eqref{PropRegu0a0ell0a1}, \eqref{PropRegu0a0ell0a2}-\eqref{PropRegu0a0ell0a3} into  \eqref{PropRegu0a0ell0a}, and using Gronwall's inequality, we immediately obtain
$$\sup_{[0,T]}\int u^2 \chi _{\varepsilon,b}(x+vt)dx\leq c_0$$
with $c_0=c_0(\varepsilon;b;v)>0$, which proves the case $\ell=0$.

Next we only consider the case $\ell=m=j+1$, because other cases are easier. 

Acting $\partial_x^{j+1}$ on the equation \eqref{hKdV} and multiplying by $\partial_x^{j+1}u(t,x)\chi_{\varepsilon,b}(x+vt)$ yield
\begin{align}
&\frac{1}{2}\frac{d}{dt} \int(\partial_x^{j+1}u)^2\chi_{\varepsilon,b}(x+vt)dx- v\int(\partial_x^{j+1}u)^2\chi'_{\varepsilon,b}(x+vt)dx\notag \\
&+\int\partial_x^{3j+2}u \partial_x^{j+1}u \chi _{\varepsilon,b}(x+vt)dx+\int\partial_x^{j+1}(u \partial_x^{j}u)\partial_x^{j+1}u \chi _{\varepsilon,b}(x+vt)dx=0 \label{PropRegu0a0}
	\end{align}
Notice that
\begin{align}
\partial_x^{3j+2}u \partial_x^{j+1}u =\frac{1}{2}\sum_{\ell=0}^{j}c_{\ell}\partial_x^{2\ell+1}\left((\partial_x^{2j+1-\ell}u)^2\right)
\label{PropRegu0a1bb}
	\end{align}
where coefficient $c_0,\cdots,c_{j}$   are gave by   \eqref{PropRegu-coeffi}. Then, substituting \eqref{PropRegu0a1bb} into \eqref{PropRegu0a0} and using integration by parts, we obtain 
\begin{align}
&\frac{1}{2}\frac{d}{dt} \int(\partial_x^{j+1}u)^2\chi_{\varepsilon,b}(x+vt)dx-\frac{c_0}{2}\int(\partial_x^{2j+1}u)^2\chi'_{\varepsilon,b}(x+vt)dx\notag \\
=& v\int(\partial_x^{j+1}u)^2\chi'_{\varepsilon,b}(x+vt)dx-\int\partial_x^{j+1}(u \partial_x^{j}u)\partial_x^{j+1}u \chi _{\varepsilon,b}(x+vt)dx\notag \\
&+\frac{1}{2}\sum_{\ell=1}^{j}c_{\ell} \int(\partial_x^{2j+1-\ell}u)^2\chi^{(2\ell+1)}_{\varepsilon,b}(x+vt)dx :=A_1+A_2+A_3.\label{PropRegu0aaa0}
	\end{align}

Integrating on the time interval $[0,T]$ and applying \eqref{lwp001}  yield
\begin{align}
\int_0^T |A_1(t)|dt&\leq v \int_0^T\int(\partial_x^{j+1}u)^2\chi'_{\varepsilon,b}(x+vt)dxdt\notag \\
&\lesssim v(b-\varepsilon)\|\partial_x^{j+1}u\|^2_{L^{\infty}_xL^2_T}<c.\label{PropRegu0aaa01}
	\end{align}

Note that
\begin{align}
A_2=&-\int \partial_x^{j+1}(u \partial_x^{j}u)\partial_x^{j+1}u \chi _{\varepsilon,b}(x+vt)dx\notag \\
=&\sum_{0\leq \ell \leq j}d_{\ell}\int \partial_x^{\ell}u( \partial_x^{j+1}u)^2 \chi^{(j-\ell)}_{\varepsilon,b}(x+vt)dx\notag \\
&+\sum_{0\leq \ell_1\leq \ell_2\leq j}d_{\ell_1,\ell_2}\int \partial_x^{j+1}u \partial_x^{\ell_1}u \partial_x^{\ell_2}u \chi^{(2j+1-\ell_1-\ell_2)}_{\varepsilon,b}(x+vt)dx\notag \\
&+\sum_{0\leq \ell_1\leq \ell_2\leq \ell_3\leq j}d_{\ell_1,\ell_2,\ell_3}\int  \partial_x^{\ell_1}u \partial_x^{\ell_2}u  \partial_x^{\ell_3}u\chi^{(3j+2-\ell_1-\ell_2-\ell_3)}_{\varepsilon,b}(x+vt)dx\notag \\
:=&\sum_{0\leq \ell \leq j}A_{2,\ell}+\sum_{0\leq \ell_1\leq \ell_2\leq j}A_{2,\ell_1,\ell_2}++\sum_{0\leq \ell_1\leq \ell_2\leq \ell_3\leq j}A_{2,\ell_1,\ell_2,\ell_3}.\label{PropRegu0aaa02a0}
	\end{align}
It is easy to see
\begin{align}
|A_{2,j}|= &\left|d_{j}\int \partial_x^{j}u( \partial_x^{j+1}u)^2 \chi _{\varepsilon,b}(x+vt)dx\right|\notag \\
\lesssim& \ \|\partial_x^{j}u\|_{L^{\infty}_x}\int (\partial_x^{j+1}u)^2 \chi _{\varepsilon,b}(x+vt)dx\label{PropRegu0aaa02a1}
	\end{align}
where the last integral is the quantity to be estimated and the term before the integral in  \eqref{PropRegu0aaa02a1} satisfies 
$$\int_0^T\|\partial_x^{j}u\|_{L^{\infty}_x}dt\leq T^{1/2}\|\partial_x^{j}u\|_{L^{2}_TL^{\infty}_x}<\infty$$
by local well-posedness theory. For $0\leq \ell\leq j-1$, 
\begin{align}
|A_{2,\ell}|\lesssim& \int \left|\partial_x^{\ell}u( \partial_x^{j+1}u)^2 \chi^{(j-\ell)}_{\varepsilon,b}(x+vt)\right|dx\notag \\
\lesssim& \ \|\partial_x^{\ell}u\|_{L^{\infty}_x}\int (\partial_x^{j+1}u)^2 \chi^{(j-\ell)}_{\varepsilon,b}(x+vt)dx\notag \\
\lesssim& \ \|u\|_{L^{\infty}_TH^{j+1/2}_x}\int (\partial_x^{j+1}u)^2 \chi^{(j-\ell)}_{\varepsilon,b}(x+vt)dx\nonumber
	\end{align}
which yields that
\begin{align}
\int_0^T|A_{2,\ell}|dt\lesssim&  \|u\|_{L^{\infty}_TH^{j+1/2}_x}\int_0^T\int (\partial_x^{j+1}u)^2 \chi^{(j-\ell)}_{\varepsilon,b}(x+vt)dx dt\notag \\
\lesssim&  \|u\|_{L^{\infty}_TH^{j+1/2}_x}\|\partial_x^{j+1}u\|_{L^{\infty}_xL^{2}_T}<\infty.
\label{PropRegu0aaa02a2}
	\end{align}

For $A_{2,\ell_1,\ell_2}$, by using H\"older's inequality,
\begin{align}
|A_{2,\ell_1,\ell_2}|\lesssim&  \int (\partial_x^{j+1}u)^2 \chi^{(j-\ell)}_{\varepsilon,b}(x+vt)dx + \int \left(\partial_x^{\ell_1}u\partial_x^{\ell_2}u\right)^2 \chi^{(j-\ell)}_{\varepsilon,b}(x+vt)dx \nonumber
	\end{align}
which deduces by integrating on time interval $[0,T]$ that
\begin{align}
\int_0^T|A_{2,\ell_1,\ell_2}|dt\lesssim&  \|\partial_x^{j+1}u\|_{L^{\infty}_xL^{2}_T}^2 + \|\partial_x^{\ell_1}u\|_{L^{\infty}_{xT}}^2 \|\partial_x^{\ell_2}u\|_{L^{\infty}_xL^{2}_T}^2 <\infty
\label{PropRegu0aaa02a3}
	\end{align}
for $0\leq \ell_1\leq \ell_2\leq j$. And, $A_{2,\ell_1,\ell_2,\ell_3}$ can be controlled in a similar way.

Finally, we consider the term $A_3$ in  \eqref{PropRegu0aaa0}. We observe that 
\begin{align}
\int_0^T|A_{3}|dt\lesssim&  \sum_{\ell=1}^{j} \int_0^T \int(\partial_x^{2j+1-\ell}u)^2\chi^{(2\ell+1)}_{\varepsilon,b}(x+vt)dx dt \notag \\
\lesssim&   \sum_{\ell=1}^{j}  \|\partial_x^{2j+1-\ell}u\|_{L^{\infty}_xL^{2}_T}^2 <\infty.
\label{PropRegu0aaa03a}
	\end{align}

Inserting the estimates  \eqref{PropRegu0aaa01}-\eqref{PropRegu0aaa03a}  into \eqref{PropRegu0aaa0} and using Gronwall’s inequality, one gets the desired result
$$ \sup_{0\leq t\leq T}\int(\partial_x^{j+1}u)^2\chi_{\varepsilon,b}(x+vt)dx+\int_0^T\int(\partial_x^{2j+1}u)^2\chi'_{\varepsilon,b}(x+vt)dxdt<c_0$$
with $c_0=c_0(\varepsilon; b; v; T)$.

\vspace{1mm}
\noindent {\bf{ Case 2 $m\geq j+2$.}} 

\vspace{1mm}

Assume that, for $m_0\geq j+1$,
\begin{align}
\sup_{0\leq t\leq T}\int_{\varepsilon-vt}^{\infty}|\partial_x^{\ell}u|^2dx+\int_0^T\int_{\varepsilon-vt}^{R+vt}(\partial_x^{m_0+j}u)^2dxdt<c 
\label{PropReguInduAssu}
	\end{align}
holds true for $0\leq \ell\leq m_0$ under the condition
$$\int_0^{\infty}|\partial_x^{m_0}u_0(x)|^2dx<\infty,$$
we need to show  \eqref{PropReguInduAssu} replacing $m_0$ by $m=m_0+1$. For simplicity, we only consider $\ell=m$.

Acting $\partial_x^{m}$ on the equation \eqref{hKdV} and multiplying by $\partial_x^{m}u(t,x)\chi_{\varepsilon,b}(x+vt)$ as before, one can get  
\begin{align}
&\frac{1}{2}\frac{d}{dt} \int(\partial_x^{m}u)^2\chi_{\varepsilon,b}(x+vt)dx- v\int(\partial_x^{m}u)^2\chi'_{\varepsilon,b}(x+vt)dx\notag \\
&+\int\partial_x^{m+2j+1}u \partial_x^{m}u \chi _{\varepsilon,b}(x+vt)dx+\int\partial_x^{m}(u \partial_x^{j}u)\partial_x^{m}u \chi _{\varepsilon,b}(x+vt)dx=0 \label{11PropRegu0m}
	\end{align}

If $m\geq 2j+1$, then
$$\int_0^T \int (\partial_x^{m}u)^2\chi'_{\varepsilon,b}(x+vt)dx dt$$
can be controlled by the assumption \eqref{PropReguInduAssu}. If $j+2\leq m\leq 2j$, one can control 
the term above by $  \|J^{j+1/2}\partial_x^{j}u\|_{L^{\infty}_xL^{2}_T}^2 $ which is bounded from the local well-posedness theory.

It follows from \eqref{PRdd2} that
\begin{align}
\partial_x^{m+2j+1}u \partial_x^{m}u =\frac{1}{2}\sum_{\ell=0}^{j}c_{\ell}\partial_x^{2\ell+1}\left((\partial_x^{m+j-\ell}u)^2\right).\nonumber
	\end{align}
Hence, after integration by part
\begin{align}
&\int_0^T \int\partial_x^{m+2j+1}u \partial_x^{m}u \chi _{\varepsilon,b}(x+vt)dxdt \notag \\
 =&-\frac{c_0}{2}\int_0^T\int  (\partial_x^{m+j}u)^2 \chi'_{\varepsilon,b}(x+vt)dx dt\notag \\
&-\frac{1}{2}\sum_{\ell=1}^{j}c_{\ell}\int_0^T\int (\partial_x^{m+j-\ell}u)^2 \chi^{(2\ell+1)}_{\varepsilon,b}(x+vt)dx dt.  \label{11PropRegu0m2}
	\end{align}
The first term in the right-hand side of \eqref{11PropRegu0m2} is what we want to estimate and  the second  term is bounded from \eqref{Chi0a3} and the assumption \eqref{PropReguInduAssu}.

To address the last term in the left-hand side of \eqref{11PropRegu0m}, we write 
\begin{align}
&\int \partial_x^{m}(u \partial_x^{j}u)\partial_x^{m}u \chi _{\varepsilon,b}(x+vt)dx\notag \\
=&\beta_{0}\int \partial_x^{j}u( \partial_x^{m}u)^2 \chi_{\varepsilon,b}(x+vt)dx\notag \\
&+\sum_{\substack{0\leq \ell_1\leq [\frac{j-1}{2}], \\ \ell_2 \geq1} }\beta_{\ell_1,\ell_2}\int \partial_x^{j-2\ell_1-\ell_2}u( \partial_x^{m+\ell_1}u)^2 \chi^{(\ell_2)}_{\varepsilon,b}(x+vt)dx\notag \\ 
&+\beta_{1}\int  \partial_x^{j+1}u \partial_x^{m-1}u  \partial_x^{m}u \chi_{\varepsilon,b}(x+vt)dx\notag \\
&+\sum_{ 2\leq \ell \leq \frac{m-j}{2}   }\beta_{\ell}\int  \partial_x^{j+\ell}u \partial_x^{m-\ell}u  \partial_x^{m}u \chi_{\varepsilon,b}(x+vt)dx\notag \\
:=&B_0+\sum_{ \ell_1\leq [\frac{j-1}{2}],  \ell_2 \geq1 }B_{\ell_1,\ell_2}+B_1+\sum_{ 2\leq \ell \leq \frac{m-j}{2}   }B_{\ell}.\label{11PropRegu0m3}
	\end{align}
Then
	\begin{align}
|B_0|\lesssim & \|\partial_x^{j}u\|_{L^{\infty}_x}\int ( \partial_x^{m}u)^2 \chi_{\varepsilon,b}(x+vt)dx \notag \\
\lesssim &\|u\|_{H^{j+1/2}_x}\int ( \partial_x^{m}u)^2 \chi_{\varepsilon,b}(x+vt)dx \nonumber
	\end{align}
where the last integral is the quantity to be estimated.

After integration in time, using  \eqref{Chi0a3} and the assumption \eqref{PropReguInduAssu} deduce that
	\begin{align}
\int_0^T|B_{\ell_1,\ell_2}|dt\lesssim & \|\partial_x^{j-2\ell_1-\ell_2}u\|_{L^{\infty}_x}\int_0^T\int ( \partial_x^{m+\ell_1}u)^2 \chi'_{\varepsilon/3,b+\varepsilon}(x+vt)dxdt<\infty.  \nonumber
	\end{align}

For the third term $B_{1}$, it follows from the identity
$$\partial_x^{m-1}u\partial_x^{m}u=\frac{1}{2}\partial_x(\partial_x^{m-1}u)^2$$
 that
	\begin{align}
\int  \partial_x^{j+1}u \partial_x^{m-1}u  \partial_x^{m}u \chi_{\varepsilon,b}(x+vt)dx=&-\frac{1}{2}\int  \partial_x^{j+2}u (\partial_x^{m-1}u)^2 \chi_{\varepsilon,b}(x+vt)dx\notag \\
&-\frac{1}{2}\int  \partial_x^{j+1}u (\partial_x^{m-1}u)^2 \chi'_{\varepsilon,b}(x+vt)dx.\label{11PropRegu0m3B3}
	\end{align}
The first term in the right-hand side of \eqref{11PropRegu0m3B3} is bounded by using the argument as $B_0$. So, we only need to estimate the second term. Applying \eqref{Chi0a4} deduces
	\begin{align}
&\int  \left|\partial_x^{j+1}u (\partial_x^{m-1}u)^2 \chi'_{\varepsilon,b}(x+vt)\right|dx \notag \\
\leq &\|\partial_x^{j+1}u\chi'_{\varepsilon/3,b+\varepsilon}(x+vt)\|_{L^{\infty}_x} \int \left|  (\partial_x^{m-1}u)^2 \chi_{\varepsilon/3,b+\varepsilon}(x+vt)\right|dx \nonumber
	\end{align}
of which the later integral is bounded by the assumption \eqref{PropReguInduAssu}.

Note that 
\begin{align}
&\|\partial_x^{j+1}u\chi'_{\varepsilon/3,b+\varepsilon}(x+vt)\|^2_{L^{\infty}_x}\notag\\
\lesssim& \|(\partial_x^{j+1}u)^2\chi'_{\varepsilon/3,b+\varepsilon}(x+vt)\|_{L^{\infty}_x}\notag \\
\lesssim&  \left|\int  \partial_x\left((\partial_x^{j+1}u)^2\chi'_{\varepsilon/3,b+\varepsilon}(x+vt)\right)dx\right| \notag \\
\lesssim & \int  \partial_x^{j+1}u\partial_x^{j+2}u\chi'_{\varepsilon/3,b+\varepsilon}(x+vt)dx +\int  (\partial_x^{j+1}u)^2\chi''_{\varepsilon/3,b+\varepsilon}(x+vt)dx \notag\\
\lesssim & \int  (\partial_x^{j+1}u)^2\chi'_{\varepsilon/3,b+\varepsilon}(x+vt)dx +\int  (\partial_x^{j+1}u)^2\chi'_{\varepsilon/3,b+\varepsilon}(x+vt)dx\notag\\
&+\int  (\partial_x^{j+1}u)^2\chi'_{\varepsilon/9,b+\varepsilon/3}(x+vt)dx. \label{11PropRegu0m3B31}
\end{align}
 Integrating on the interval $[0, T]$ and applying the assumption \eqref{PropReguInduAssu}, from \eqref{11PropRegu0m3B31} one can control 
$$\|\partial_x^{j+1}u\chi'_{\varepsilon/3,b+\varepsilon}(x+vt)\|_{L^{\infty}_x}$$
which by \eqref{11PropRegu0m3B3} further  implies that $$\int_0^T|B_{1}|dt<\infty.$$

Finally, let us consider $B_{\ell}$. By H\"older's inequality, we get
	\begin{align}
|B_{\ell}|\leq \int  (\partial_x^{j+\ell}u \partial_x^{m-\ell}u)^2 \chi_{\varepsilon,b}(x+vt)dx + \int (\partial_x^{m}u)^2\chi_{\varepsilon,b}(x+vt)dx  \label{11PropRegu0m3B4}
	\end{align}
where the last integral is what to be estimated.  

Observe that
$$\chi_{\varepsilon/5,\varepsilon}(x)=1 \hspace{5mm} \text{on} \hspace{2mm} \text{supp} \chi_{\varepsilon,b} \subset [\varepsilon, \infty),$$
hence 
\begin{align}
 &\int  (\partial_x^{j+\ell}u \partial_x^{m-\ell}u)^2 \chi_{\varepsilon,b}(x+vt)dx \notag \\
\lesssim &  \|(\partial_x^{j+\ell}u)^2 \chi_{\varepsilon/5,\varepsilon}(x+vt)\|_{L^{\infty}_x}\int  (\partial_x^{m-\ell}u)^2 \chi_{\varepsilon,b}(x+vt)dx \label{11PropRegu0m3B41}
	\end{align}
of which the last integral is bounded by  induction \eqref{PropReguInduAssu} for $2\leq \ell \leq \frac{m-j}{2} $. According to Sobolev embedding,
\begin{align}
 &\int_0^T \|(\partial_x^{j+\ell}u)^2 \chi_{\varepsilon/5,\varepsilon}(x+vt)\|_{L^{\infty}_x} dt\notag \\
\lesssim & \int_0^T \left\|\partial_x\left((\partial_x^{j+\ell}u)^2 \chi_{\varepsilon/5,\varepsilon}(x+vt)\right)\right\|_{L^{1}_x}dt \notag \\
\lesssim &  \int_0^T\left\|\partial_x^{j+\ell}u\partial_x^{j+\ell+1}u \chi_{\varepsilon/5,\varepsilon}(x+vt)\right\|_{L^{1}_x}dt+ \int_0^T\left\|(\partial_x^{j+\ell}u)^2 \chi'_{\varepsilon/5,\varepsilon}(x+vt)\right\|_{L^{1}_x}dt\notag \\
\lesssim_T &  \sup_{0\leq t\leq T}\int (\partial_x^{j+\ell}u)^2 \chi_{\varepsilon/5,\varepsilon}(x+vt)dx+\sup_{0\leq t\leq T}\int (\partial_x^{j+\ell+1}u)^2 \chi_{\varepsilon/5,\varepsilon}(x+vt)dx\notag \\
&+ \int_0^T\int(\partial_x^{j+\ell}u)^2 \chi'_{\varepsilon/5,\varepsilon}(x+vt)dxdt<\infty \label{11PropRegu0m3B42}
	\end{align}
by  induction \eqref{PropReguInduAssu}  provided that  $2\leq \ell \leq \frac{m-j}{2} $.
\eqref{11PropRegu0m3B4}-\eqref{11PropRegu0m3B42} implies the desired result.
 
Then, substituting all these estimates into \eqref{11PropRegu0m} gives \eqref{PropReguInduAssu} for $m\geq j+2$ which completes the proof of Theorem \ref{prop-reg}.\hspace{84.5mm}$\square$

%%%%%%%%%%%%%%%%%%%%%%%%%%%%%%%%%%%%%%%%%%%%%%%%%%%%%%%%%%%%%%%%%%%%%%
\section{Dispersive blow up} 

We investigate dispersive blow-up solutions  to \eqref{hKdV} in this section. In subsection \ref{Linearsingu}, we construct an initial datum  $u_0\in C^\infty(\mathbb{R}) $  such that the free solution  $W(t)u_0$ fails to be in $C^{j+1}(\mathbb{R})$ at all positive rational time. In subsection \ref{NonlSmo}, we show that the Duhamel term is smoother than the linear part. To be precise, we show that $z_1(t)\in H^{j+\frac{3}{2}+}(\mathbb{R})$ and therefore embedded in $C^{j+1}	(\mathbb{R})$. This suggests that the blow-up phenomenon appears due to the linear component of the solution.

\subsection{Linear singularities}\label{Linearsingu}

Choosing
$$\varphi(x)=e^{-2|x|^{j+1}},$$
one can easily verify that
$$e^{x}\varphi(x) \in L^2(\mathbb{R}), \hspace{8mm} \varphi(x) \in C^{\infty}(\mathbb{R}\backslash0)\backslash C^{j+1}(\mathbb{R}).$$
and  $\varphi(x)\in Z_{s,r}(\mathbb{R})$ for any $s\in\left[j+1, j+3/2\right)$.
 
We would like to consider the regularity of $\varphi $ under the higher-order linear KdV flow.

\begin{lemma}  \label{linbuplema}
Let $\varphi=e^{-2|x|^{j+1}}$. Then,  
\begin{align}
\left\|\partial^{\ell}_{x} e^{(-1)^{j+1}x}W(t)\varphi\right\|_{L^2_{x}}
		\lesssim   t^{-\frac{\ell}{2}} e^{t} \left\|e^{(-1)^{j+1}x}\varphi\right\|_{L^2}  
		<  \infty  \label{linbuplema1}
	\end{align}
for $t>0$, and 
\begin{align}
\left\|\partial^{\ell}_{x} e^{(-1)^{j}x}W(t)\varphi\right\|_{L^2_{x}}
		\lesssim   |t|^{-\frac{\ell}{2}} e^{-t} \left\|e^{(-1)^{j}x}\varphi\right\|_{L^2}  
		<  \infty   \label{linbuplema2}
	\end{align}
for $t<0$, where $\ell=0,1,\cdots$. As a consequence, one has $W(t)\varphi\in C^{\infty}(\mathbb{R}) $ for $t\neq0$.
\end{lemma}
\begin{proof} It follows from Sobolev embedding theorem that 
\begin{align}
W(t)\varphi\in C^{\infty}(\mathbb{R}) 
\Longleftrightarrow& \hspace{1mm} e^{\pm x}W(t)\varphi\in C^{\infty}(\mathbb{R})	\notag \\
\Longleftrightarrow&  \hspace{1mm} e^{\pm x}W(t)\varphi\in H^{\ell}(\mathbb{R}) \hspace{3mm} \text{for all}  \hspace{2mm} \ell\in \mathbb{N}.	\nonumber	
\end{align}
So, it suffices to show \eqref{linbuplema1}. For simplicity, we only consider the case that $j$ is an odd number.

 Put $v(t)=W(t)\varphi$. It is easy to see that $v(t)$ is solution to the following linear equation
\begin{equation}
	\left\{
	\begin{aligned}
		&\partial_t v+\partial^{2j+1}_x v=0 , \quad \\
		&v(0,x)=\varphi(x).  \label{linear-hKdV} \\
	\end{aligned}
	\right.
\end{equation}
Denote  $w(t)=e^{x}v(t)$. Putting $v(t)=e^{-x}w(t)$  into  \eqref{linear-hKdV}, we see that $w(t)$ is   solution to
\begin{equation}
	\left\{
	\begin{aligned}
		&\partial_t w+(\partial_x-1)^{2j+1} w=0 , \quad \\
		&w(0,x)=e^{x}\varphi(x).  \nonumber \\
	\end{aligned}
	\right.
\end{equation}
By Fourier transform, one has
\begin{equation*}
	\widehat{w}=e^{-t(i\xi-1)^{2j+1}}\widehat{w_0}
\end{equation*}
which further implies by using Plancheral's identity that
\begin{equation*}
	\begin{aligned}
	\left\|\partial^{\ell}_{x} w\right\|_{L^2}
		= \big\||\xi|^{\ell}\widehat{w}\big\|_{L^2}
		=  &\left\||\xi|^{\ell}e^{-t(i\xi-1)^{2j+1}}\widehat{w_0}\right\|_{L^2}\\
		\lesssim &e^{t}\big\||\xi|^{\ell}e^{-t(\xi^2-\xi^4+\cdots+\xi^{2j})}\big\|_{L^{\infty}}	\|\widehat{w_0}\|_{L^2}\\
\lesssim &e^{t}\big\||\xi|^{\ell}e^{-t\xi^2}\big\|_{L^{\infty}}\left\|e^{x}\varphi\right\|_{L^2}\\
		\lesssim  & |t|^{-\frac{\ell}{2}} 	e^{t}\left\|e^{x}\varphi\right\|_{L^2}
		< \infty.
	\end{aligned} 	  	
\end{equation*}
This completes the proof.
\end{proof}
In the next place, we use the  function $\varphi(x)$ to construct a smooth initial datum $u_0$ such that the linear solution $W(t)u_0$ of \eqref{hKdV}  will display singularity  at each time-space positive rational point.

\begin{theorem}  \label{linbupthm}
Assume that
\begin{align}
	u_0=\sum_{\substack{p_2,q_2\in\mathbb{Z}^+,\\ gcd(p_2,q_2)=1}} \sum_{\substack{p_1,q_1\in\mathbb{Z}^+,\\ gcd(p_1,q_1)=1}}e^{-e^{(q_1+q_2)}}e^{-(p_1^2+p_2^2)} W\big(-\frac{p_2}{q_2}\big)\varphi\big(x-\frac{p_1}{q_1}\big)  \label{initval000}
\end{align}
where $\varphi(x)=e^{-2|x|^{j+1}}$ , then  we have 
\begin{equation*}
	\left\{\begin{array}{ll}
		W(t)u_0\in C^{\infty}	(\mathbb{R}),    &t>0, \ t\in \mathbb{R}\setminus\mathbb{Q}^*,\\
		W(t)u_0\in C^{\infty}	(\mathbb{R}\backslash \mathbb{Q}^+)\backslash C^{j+1}(\mathbb{R}),  \ \  &t>0, \ t\in \mathbb{Q}\subset\mathbb{Q}^*.
	\end{array}\right.	
\end{equation*}
\end{theorem}
\begin{proof} Without loss of generality, one can assume that $j$ is odd.
  
Firstly,  according to Lemma  \ref{linbuplema} and \eqref{initval000}, we get
	\begin{align}
		\left\|\partial^{\ell}_{x} e^{-x}u_0\right\|_{L^2}
		\lesssim &\sum_{\substack{p_2,q_2\in\mathbb{Z}^+,\\ gcd(p_2,q_2)=1}} \sum_{\substack{p_1,q_1\in\mathbb{Z}^+,\\ gcd(p_1,q_1)=1}}e^{-e^{q_2}}e^{-p_2^2}  \left\|\partial^{\ell}_{x} e^{-x}W\big(-\frac{p_2}{q_2}\big)\varphi\big(x-\frac{p_1}{q_1}\big)\right\|_{L^2} \notag\\
		\lesssim & \sum_{\substack{p_2,q_2\in\mathbb{Z}^+,\\ gcd(p_2,q_2)=1}} \sum_{\substack{p_1,q_1\in\mathbb{Z}^+,\\ gcd(p_1,q_1)=1}}e^{-e^{q_2}}e^{-p_2^2}  p_2^{-\frac{\ell}{2}}q_2^{\frac{\ell}{2}} e^{\frac{p_2}{q_2}}e^{-\frac{p_1}{q_1}} \left\|e^{-x}\varphi\right\|_{L^2}  <\infty. \nonumber
	\end{align}
By Sobolev embedding theorem, we see that
$u_0\in C^{\infty}	(\mathbb{R})$.

Secondly,
\begin{align}
W(t)	u_0=\sum_{\substack{p_2,q_2\in\mathbb{Z}^+,\\ gcd(p_2,q_2)=1}} \sum_{\substack{p_1,q_1\in\mathbb{Z}^+,\\ gcd(p_1,q_1)=1}}e^{-e^{(q_1+q_2)}}e^{-(p_1^2+p_2^2)} W\big(t-\frac{p_2}{q_2}\big)\varphi\big(x-\frac{p_1}{q_1}\big),\nonumber
	\end{align}
using the same argument as above, one  obtains that
  $W(t)u_0\in C^{\infty}(\mathbb{R})$ for $t\in \mathbb{R}\setminus\mathbb{Q}^*$ and $t>0$. 

Here is the reason why we introduce generic irrational number (see Definition \ref{gene-irrat}). From \eqref{linbuplema1}, we use $|t-\frac{p_2}{q_2}|^{-\frac{\ell}{2}}$ to control $\left\|\partial^{\ell}_{x} e^{-x}W(t-\frac{p_2}{q_2})u_0\right\|_{L^2}$. If $t$ is a irrational number that is quickly approximated by rational numbers, then $|t-\frac{p_2}{q_2}|^{-\frac{\ell}{2}}$ is of singularity. However, this  will not occur for generic irrational number. Because,  one has
$$\big|t-\frac{p_2}{q_2}\big|^{-\frac{\ell}{2}}\lesssim (|p_2|+|q_2|)^{\frac{3\ell}{2}} $$
from \eqref{generic2}.

Finally,  for $t=p/q\in\mathbb{Q}^+$,
	\begin{align}
	W\big(\frac{p}{q}\big)u_0=&\mathop{\sum}_{(p_2,q_2)\neq(p,q)} \sum_{(p_1,q_1)}e^{-e^{(q_1+q_2)}}e^{-(p_1^2+p_2^2)} W\big(\frac{p}{q}-\frac{p_2}{q_2}\big)\varphi\big(x-\frac{p_1}{q_1}\big) \notag \\
 &+ \sum_{(p_1,q_1)}e^{-e^{(q_1+q)}}e^{-(p_1^2+p^2)} \varphi\big(x-\frac{p_1}{q_1}\big).	 \label{rationaltime} 	
\end{align}
The first summation on right-hand side of \eqref{rationaltime} is in $C^\infty(\mathbb{R})$, but the second summation is in $C^{\infty}(\mathbb{R}\backslash \mathbb{Q}^+)\backslash C^{j+1}(\mathbb{R})$.  Hence, $W\big(\frac{p}{q}\big)u_0\in C^{\infty}(\mathbb{R}\backslash \mathbb{Q}^+)\backslash C^{j+1}(\mathbb{R})$.
\end{proof}
%%%%%%%%%%%%%%%%%%%%%%%%%%%%%%%%%%%%%%%%%%%%%%%%%%%%%%%%%%%%%%%%

\subsection{Nonlinear smoothing}\label{NonlSmo}
This subsection devotes to show Theorem \ref{dispersiveblow-upsolution}, Theorem \ref{ghKdVDuhamelkg2}  and Theorem \ref{DoNotProg}.

\subsubsection{Nonlinear smoothing for $k=1$}
\textbf{Proof of Theorem \ref{dispersiveblow-upsolution}. }
We will show that the Duhamel term defined as
 $$ z_1(t)=\int^t_0 W(t-t')(u\partial^j_x u)(t')\mathrm{d}t'$$ 
belongs to $H^{j+\frac32+}(\mathbb{R})$
for initial value 
\begin{align}
u_0\in \bigcap\limits_{s\in \left[j+1,j+\frac32\right)} Z_{s,r}\nonumber
\end{align}
where $0<r<1$. This  implies in particular $z_1(t)\in C^{j+1}(\mathbb{R})$ by Sobolev embedding theorem.

Now we begin to estimate $\Big\|D_x^{j+\frac32+}z_1\Big\|_{L^2_{xy}}$. Applying the dual version of the smoothing effect \eqref{linear estimate:Kato-2a} and Cauchy-Schwarz inequality, we have
\begin{align}
	\left\|D_x^{j+\frac32+}\int^t_0 W(t-t')u\partial^j_x udt'\right\|_{L^2_{x}}
	\lesssim&  \big\|D_x^{\frac{3}{2}+}u\partial^j_x u\big\|_{L^1_xL^2_{T}}\notag \\
	\lesssim& \big\|uD_x^{\frac{3}{2}+} \partial^j_x u\big\|_{L^1_xL^2_{T}}+\big\|[D_x^{\frac{3}{2}+},u]\partial^j_x u\big\|_{L^1_xL^2_{T}}
:= T_1+T_2.	\nonumber
\end{align}

%%%%%%%%%%%%%%%%%%%%%%%%%%%%%%%%%%%%%%%%%%%%%%%%%%%%%%%%%%%
\noindent
$\bullet$ {\bf Estimate for  $T_1$.}  
\vspace{1mm}

By H\"older's inequality,  we have
\begin{align}
T_1= \big\|uD_x^{\frac{3}{2}+} \partial^j_x u\big\|_{L^1_xL^2_{T}}\lesssim \|u\|_{L^{6/5}_xL^3_{T}}\big\|D_x^{\frac{3}{2}+} \partial^j_x u\big\|_{L^6_xL^6_{T}}.\label{nonliT1a0}
\end{align}

It follows from Duhamel's principle and Strichartz estimate \eqref{ineq:strichartz1} with $p=q=6$, $\theta=2/3$ and $\tilde{p}=2$, $\tilde{q}=\infty$ that
\begin{align}
\big\|D_x^{\frac{3}{2}+} \partial^j_x u\big\|_{L^6_xL^6_{T}}&\lesssim \big\|D_x^{\frac{3}{2}+} \partial^j_x W(t)u_0\big\|_{L^6_xL^6_{T}}+\left\| D_x^{\frac{3}{2}+} \partial^j_x \int_0^tW(t-t')(u \partial^j_x u)(t')dt'\right\|_{L^6_xL^6_{T}}\notag \\
&\lesssim \|u_0\|_{H^{s}}+\int_0^T\big\|J^{j+\frac{3}{2}-\frac{2j-1}{6}+}(u\partial^j_x u)\big\|_{L^2_{x}}dt\notag \\
&\lesssim \|u_0\|_{H^{s}}+\|u\|^2_{X_T}.\label{nonliT1a1}
\end{align}
The veracity of last step in \eqref{nonliT1a1} follows readily from local well-posedness results.

By using H{\"o}lder's inequality, Sobolev's inequality and  interpolation inequality  \eqref{lem:interpolaion2}, we deduce
\begin{align}
\|u\|_{L^{6/5}_xL^3_{T}} \lesssim  \|\langle x\rangle^{\frac{1}{2}+}u\|_{L^{3}_{xT}}&\lesssim \big\| \|\langle x\rangle^{\frac{1}{2}+}u\|_{L^{3}_{x}}\big\|_{L^{3}_{T}}\notag \\
&\lesssim \big\| \|J^{1/6}\langle x\rangle^{\frac{1}{2}+}u\|_{L^{2}_{x}}\big\|_{L^{3}_{T}}\notag \\
&\lesssim T^{1/3} \|J^{\frac{1}{6\gamma}}u\|^{\gamma}_{L^{\infty}_{T}L^{2}_{x}}\|\langle x\rangle^{\frac{1}{2(1-\gamma)}+}u\|^{1-\gamma}_{L^{\infty}_{T}L^{2}_{x}}\notag \\
&\lesssim T^{1/3} \|u\|_{X_T}    \label{nonliT1a2}
\end{align}
via taking $\gamma=\frac{1}{6(j+1)}$ such that $\frac{1}{6\gamma}=j+1\leq s$ and $\frac{1}{2(1-\gamma)}=\frac{3(j+1)}{6j+5}<\frac{j+1}{2j}$.

Collecting \eqref{nonliT1a0}-\eqref{nonliT1a2}, we get
\begin{align}
T_1 \lesssim T^{1/3} \|u\|_{X_T} \left(\|u_0\|_{H^{s_0}}+\|u\|^2_{X_T}\right)<\infty. \label{nonliT1aEnd}
\end{align}

%%%%%%%%%%%%%%%%%%%%%%%%%%%%%%%%%%%%%%%%%%%%%%%%%%%%%%%%%%%
\noindent
$\bullet$ {\bf Estimate for  $T_2$.}  
\vspace{1mm}

Applying H{\"o}lder's inequality and weighted Kato-Ponce inequality \eqref{Kato-Ponce-weight1}, one has
\begin{align}
&T_2=\big\|[D_x^{\frac32+},u]\partial^j_x u\big\|_{L^1_xL^2_{T}} \notag \\
\leq& \big\|\langle x\rangle^{\frac{1}{2}+}[D_x^{\frac32+},u]\partial^j_x u\big\|_{L^2_{xT}}	\notag \\
  \lesssim & T^{\frac{1}{4}}\big\| \langle x\rangle^{\frac{\tilde{p}-2}{2\tilde{p}}+} D_x^{\frac32+}u \big\|_{L^{2\tilde{p}}_{T}L^{\frac{2\tilde{p}}{\tilde{p}-2}}_x} \big\|\langle x\rangle^{\frac{1}{\tilde{p}}+} \partial^j_x u\big\|_{L^{\tilde{q}}_{T}L^{\tilde{p}}_{x}}\notag \\
&+ T^{\frac{1}{4}}\big\| \langle x\rangle^{\frac{p-2}{2p}+}\partial_x u \big\|_{L^{2p}_{T}L^{\frac{2p}{p-2}}_x} \big\|\langle x\rangle^{\frac{1}{p}+} D_x^{\frac12+}\partial^j_x u\big\|_{L^q_{T}L^p_{x}},
\label{T2fraaa}
   \end{align}
where $(p,q)$ and $(\tilde{p},\tilde{q})$ are Strichartz pairs satisfying $4/q+2/p=4/\tilde{q}+2/\tilde{p}=1$ and $2<p, \ \tilde{p}<\infty$. We only estimate the second term in the right-hand side of \eqref{T2fraaa}, as  the first term can be dealt with in a similar way.

\vspace{1mm}
Let us now estimate $\big\| \langle x\rangle^{\frac{1}{p}+} D_x^{\frac1 2+}\partial^j_x u\big\|_{L^q_{T}L^p_x}$. Using interpolation inequality  \eqref{lem:interpolaion1} derivatives
\begin{equation*}
	\begin{aligned}
		\big\|\langle x\rangle^{\frac{1}{p}+} D_x^{\frac1 2+}\partial^j_xu\big\|_{L^p_{x}}
		\lesssim & \left\|\langle x\rangle^{\frac{1}{p\beta}+} u\right\|_{L^p_{x}}^\beta 
\left\|D_x^{\frac{(2j+1)}{2(1-\beta)}+}u\right\|_{L^p_{x}}^{1-\beta}\notag \\
		\lesssim  &\left\|\langle x\rangle^{\frac{1}{p\beta}+} u\right\|_{L^p_{x}}+
\left\|D_x^{\frac{(2j+1)}{2(1-\beta)}+}u\right\|_{L^p_{x}}.
	\end{aligned}
\end{equation*}
Hence,
	\begin{align}
	\big\|\langle x\rangle^{\frac{1}{p}+} D_x^{\frac1 2+}\partial^j_xu\big\|_{L^{q}_{T}L^p_{x}}
		\lesssim 
\left\|\langle x\rangle^{\frac{1}{p\beta}+} u\right\|_{L^{q}_{T}L^p_{x}}+
\left\|D_x^{\frac{(2j+1)}{2(1-\beta)}+}u\right\|_{L^{q}_{T}L^p_{x}}.\label{T2a1aa}
	\end{align}

For the first term, we use Sobolev's inequality and interpolation inequality 
	\begin{align}
		\big\|\langle x\rangle^{\frac{1}{p\beta}+} u\big\|_{L^p_{x}}
		&\lesssim \big\|J^{\frac12-\frac{1}{p}}\langle x\rangle^{\frac{1}{p\beta}+} u\big\|_{L^2_{x}} \notag\\
		&\lesssim \big\|J^{\frac{p-2}{2p(1-\sigma)}} u\big\|_{L^2_{x}}^{1-\sigma} \big\|\langle x\rangle^{\frac{1}{p\beta\sigma}+} u\big\|_{L^2_{x}}^{\sigma}\notag\\
		&\lesssim \big\|J^{\frac{p-2}{2p(1-\sigma)}} u\big\|_{L^2_{x}}+ \big\|\langle x\rangle^{\frac{1}{p\beta\sigma}+} u\big\|_{L^2_{x}}.\label{T2a1a}
	\end{align}
Put 
\begin{align}
\beta= \frac{4p(s-1)-(p+2)(2j-1)}{4ps+(p-2)(2j-1)}-, \hspace{5mm} \sigma= \frac{2j}{p\beta(j+1)}+\label{nonlinindex0}
\end{align}
such that 
\begin{align}
\frac{(2j+1)}{2(1-\beta)}+= s+\frac{(p-2)(2j-1)}{4p},   \hspace{7mm} \frac{1}{p\beta\sigma}=\frac{j+1}{2j}-=r-<\frac{s}{2j}. \nonumber\end{align}
By choosing $p$ large enough, we see from \eqref{nonlinindex0} that 
$$\frac{1}{8}<\beta<1, \hspace{3mm} \text{and} \hspace{3mm} 0<\sigma<\frac{1}{10}.$$
Hence,  $\frac{p-2}{2p(1-\sigma)}<\frac{5}{9}$, then by \eqref{T2a1a} one gets
	\begin{align}
		\big\|\langle x\rangle^{\frac{1}{p\beta}+} u\big\|_{L^{q}_{T}L^p_{x}}
		\lesssim&  T^{\frac{1}{q}}\left(\big\|J^{\frac{p-2}{2p(1-\sigma)}} u\big\|_{L^\infty_T L^2_{x}} +\big\|\langle x\rangle^{\frac{1}{p\beta\sigma}+} u\big\|_{L^\infty_T L^2_{x}}\right) \notag\\
	\lesssim&  T^{\frac{1}{q}}\left(\| u\|_{L^\infty_T H^{5/9}_{x}} +\big\|\langle x\rangle^{r} u\big\|_{L^\infty_T L^2_{x}}\right) 
\lesssim T^{\frac{1}{q}}\|u\|_{X_T}. \label{T2a1b}
	\end{align}

For the second term in the right hand side of \eqref{T2a1aa}, by using Lemma \ref{linear estimate:strichartz}  with $\theta=\frac{p-2}{p}$,   we derive 
\begin{align}
	&\left\|D_x^{\frac{(2j+1)}{2(1-\beta)}+}u\right\|_{L^{q}_{T}L^p_{x}}=\big\|D_x^{s+\frac{\theta(2j-1)}{4}}u\big\|_{L^{q}_{T}L^p_{x}} \notag\\
	\lesssim& 	\big\|D_x^{s+\frac{\theta(2j-1)}{4}}W(t)u_0\big\|_{L^{q}_{T}L^p_{x}}+\left\|D_x^{s+\frac{\theta(2j-1)}{4}}\int_0^T W(t-\tau)(u\partial^j_x u)\mathrm{d}\tau\right\|_{L^{q}_{T}L^p_{x}}\notag\\
	\lesssim & \big\|D_x^{s} u_0\big\|_{L^2_{x}}+\int_0^T\big\|D_x^{s} 
(u\partial^j_x u)\big\|_{L^2_{x}}dt\notag\\
 	\lesssim  &\| u_0\|_{H^{s}}+\|u\|_{X_T}^2.\label{T2a1c}
\end{align}

Combining \eqref{T2a1aa}, \eqref{T2a1b} and \eqref{T2a1c},  we obtain 
	\begin{align}
		\big\|\langle x\rangle^{\frac{1}{p}+} D_x^{\frac1 2+}\partial^j_xu\big\|_{L^{q}_{T}L^p_{x}}
		\lesssim_T \| u_0\|_{H^{s}}+\|u\|_{X_T}+\|u\|_{X_T}^2<\infty.\label{T2aa1a1A}
	\end{align}

Additionally,  from  interpolation inequality  \eqref{lem:interpolaion1} and  Sobolev's inequality, one sees that 
	\begin{align}
&\big\| \langle x\rangle^{\frac{p-2}{2p}+} \partial_x u \big\|_{L^{2p}_{T}L^{\frac{2p}{p-2}}_x}\notag \\
\lesssim&  \big\| \langle x\rangle^{\frac{p-2}{2p\alpha}+} u \big\|^{\alpha}_{L^{2p}_{T}L^{\frac{2p}{p-2}}_x} 
\big\|  D_x^{\frac{1}{1-\alpha}}u \big\|^{1-\alpha}_{L^{2p}_{T}L^{\frac{2p}{p-2}}_x}\notag \\
\lesssim&  \big\| \langle x\rangle^{\frac{p-2}{2p\alpha}+} u \big\|_{L^{2p}_{T}L^{\frac{2p}{p-2}}_x}+ 
\big\|  D_x^{\frac{1}{1-\alpha}}u \big\|_{L^{2p}_{T}L^{\frac{2p}{p-2}}_x}\notag \\
\lesssim&  T^{1/2p}\big\| J^{\frac{1}{p}}\langle x\rangle^{\frac{p-2}{2p\alpha}+} u \big\|_{L^{\infty}_{T}L^{2}_x}+ 
\big\|  D_x^{\frac{1}{1-\alpha}}u \big\|_{L^{2p}_{T}L^{\frac{2p}{p-2}}_x}.\label{T2bb2b1}
	\end{align}
Using again  interpolation inequality deduces that
	\begin{align}
\big\| J^{\frac{1}{p}}\langle x\rangle^{\frac{p-2}{2p\alpha}+} u \big\|_{L^{\infty}_{T}L^{2}_x}\lesssim  \big\| J^{\frac{1}{p(1-\tilde{\alpha})}} u \big\|_{L^{\infty}_{T}L^{2}_x}+\big\| \langle x\rangle^{\frac{p-2}{2p\alpha\tilde{\alpha}}+} u \big\|_{L^{\infty}_{T}L^{2}_x}.\label{T2bb2b1a}
	\end{align}

By choosing $\alpha=1-\frac{2p}{2ps+2j-1}$ and $\tilde{\alpha}=\frac{(p-2)j}{p\alpha(j+1)}+$, such that 
$$ \frac{1}{p(1-\tilde{\alpha})}\leq s, \hspace{5mm} 
\frac{p-2}{2p\alpha\tilde{\alpha}}=\frac{j+1}{2j}-=r-, \hspace{3mm} \text{and} \hspace{3mm} \frac{1}{1-\alpha}=s+\frac{2j-1}{2p},$$ then we have
	\begin{align}
 \big\| J^{\frac{1}{p(1-\tilde{\alpha})}} u \big\|_{L^{\infty}_{T}L^{2}_x}
<\|u\|_{L^{\infty}_{T}H^s}, \hspace{5mm}
\big\| \langle x\rangle^{\frac{p-2}{2p\alpha\tilde{\alpha}}+} u \big\|_{L^{\infty}_{T}L^{2}_x}<\|u\|_{X_T}, \label{T2bb2b2}
	\end{align}
and 
	\begin{align}
\big\|  D_x^{\frac{1}{1-\alpha}}u \big\|_{L^{2p}_{T}L^{\frac{2p}{p-2}}_x}=\big\|  D_x^{s+\frac{2j-1}{2p}}u \big\|_{L^{2p}_{T}L^{\frac{2p}{p-2}}_x}\lesssim  &\| u_0\|_{H^{s}}+\|u\|_{X_T}^2.\label{T2bb2b3}
	\end{align}
The last inequality above is from Strichartz  estimate and Duhamel's principle, see also \eqref{T2a1c}.

So,  \eqref{T2bb2b1}-\eqref{T2bb2b3} yield 
\begin{align}
		\big\| \langle x\rangle^{\frac{p-2}{2p}+} \partial_x u \big\|_{L^{2p}_{T}L^{\frac{2p}{p-2}}_x}
		\lesssim_T \| u_0\|_{H^{s}}+\|u\|_{X_T}+\|u\|_{X_T}^2<\infty.\label{T2bb2bBBB}
	\end{align}
Therefore, we deduce from  \eqref{T2fraaa}, \eqref{T2aa1a1A} and \eqref{T2bb2bBBB} that
\begin{align}
T_2 \lesssim_T   \|u_0\|^2_{H^{s}}+\|u\|^2_{X_T}+\|u\|^4_{X_T} <\infty. \label{nonliT1aEnd}
\end{align}
We finish the proof of Theorem \ref{dispersiveblow-upsolution}.\hspace{86mm}$\square$

%%%%%%%%%%%%%%%%%%%%%%%%%%%%%%%%%%%%%%%%%%%%%%%%%%%
\subsubsection{Nonlinear smoothing for $k\geq2$}

\textbf{Proof of Theorem \ref{ghKdVDuhamelkg2}. } 
The arguments utilized to show this theorem  are
the local smoothing effect and maximal function estimates. We only deal with the case $s=2$, because  techniques we used here are applicable to larger $s$. Observe that $z_k\in L^2(\mathbb{R}^2)$, it suffices to control the $L^2$ norms of 
$$\partial_x^{s+j}\int_0^t W(t-t')u^k\partial^j_xu(t')dt' .$$ 

Applying ﻿the dual version of the smoothing effect \eqref{linear estimate:Kato-2} and H\"older's inequality deduce
\begin{align}
 &\left\|\partial_x^{s+j}\int_0^T W(t-t')u^k\partial^j_xu(t')dt' \right\|_{L^2_{x}} 
\lesssim   \big\|\partial_x^{s}(u^k\partial^j_xu)\big\|_{L^1_xL^2_{T}} \notag \\
\lesssim & \|u^{k}\partial_x^{s+j}u\|_{L^1_xL^2_{T}}+\|\partial_x^{j}u \partial_x^{s}u^{k}\|_{L^1_xL^2_{T}}+\sum_{1\leq\ell\leq s-1}\big\|\big(\partial_x^{s-\ell}u^k\big)\big(\partial_x^{j+\ell}u\big)\big\|_{L^1_xL^2_{T}} \notag \\
=:&T_1+T_2+T_3.\nonumber
\end{align}
In the next place, we only consider $T_1$ and $T_2$,  as  $T_3$ can be estimated by using the same argument.

It follows from H\"older's inequality and Sobolev embedding that
\begin{align}
T_1=\|u^{k}\partial_x^{s+j}u\|_{L^1_xL^2_{T}}&\leq \|u^{k}\|_{L^1_xL^{\infty}_{T}}\|\partial_x^{s+j}u\|_{L^{\infty}_xL^2_{T}} \notag \\
&\leq \|u\|^{k}_{L^{k}_xL^{\infty}_{T}}\|\partial_x^{s+j}u\|_{L^{\infty}_xL^2_{T}}\notag \\
&\leq \|J^{1/2-1/k}u\|^{k}_{L^{2}_xL^{\infty}_{T}}\|\partial_x^{s+j}u\|_{L^{\infty}_xL^2_{T}}\notag \\
&\leq \|u\|^{k}_{L^{\infty}_{T}H^{s}_x}\|\partial_x^{s+j}u\|_{L^{\infty}_xL^2_{T}}\nonumber
\end{align}
which is bounded by local well-posedness theory.

For $T_2$, we only estimate the term  $\|u^{k-1}\partial_x^{s}u \partial_x^{j}u\|_{L^1_xL^2_{T}}$, because other terms can be dealt with in a similar way.  By using H\"older's inequality, Sobolev embedding and local well-posedness theory,  we obtain
\begin{align}
\|u^{k-1}\partial_x^{s}u \partial_x^{j}u\|_{L^1_xL^2_{T}}&\leq \|u^{k-1}\|_{L^{2}_xL^{\infty}_{T}}
\|\partial_x^{s}u\|_{L^{4}_xL^4_{T}}\|\partial_x^{j}u\|_{L^{4}_xL^4_{T}} \notag \\
&\leq \|u\|^{k-1}_{L^{2(k-1)}_xL^{\infty}_{T}}
\|\partial_x^{s}u\|_{L^4_{xT}}\|\partial_x^{j}u\|_{L^4_{xT}} \notag \\
&\leq T^{1/4}\|J^{\frac{1}{2}-\frac{1}{2(k-1)}}u\|^{k-1}_{L^{2}_xL^{\infty}_{T}}\|\partial_x^{s}u\|_{L^{8}_tL^4_{x}}\|\partial_x^{j}u\|_{L^{8}_tL^4_{x}} <\infty. \nonumber
\end{align}
Therefore, we finish the proof. \hspace{93.5mm}$\square$

%%%%%%%%%%%%%%%%%%%%%%%%%%%%%%%%%%%%%%%%%%%%%%%%%%%%%%%%%%%%%%%%%%%%%
In the end,  we show that the singularities of solutions to the higher KdV equation \eqref{hKdV} do not propagate in any direction.

\textbf{Proof of Theorem \ref{DoNotProg}. } From Theorem \ref{ghKdVDuhamelkg2}, we see that the nonlinear part of the solution is in $H^{2j+1}(\mathbb{R})$ which yields that 
$z_k(t)\in W^{j+1,p}(\mathbb{R})$ for any $p>2$ by  Sobolev embedding theorem.

Choosing $\phi \in H^{j+1}(\mathbb{R})\cap W^{j+1,1}(\mathbb{R})$ and $\phi \notin  W^{j+1,p}(\mathbb{R})$ for every $p>2$, by using the dispersive estimate \eqref{hKdVdispEs}, we have
$$ \|W(t)\partial_x^{j+1} \phi\|_{L_x^{\infty}}\lesssim |t|^{-1/2}\|\partial_x^{(2j+1)/4} \phi\|_{L_x^{1}}$$
which further implies $W(t)\phi\in W^{j+1,p}(\mathbb{R})$ for any $p\geq2$ via an interpolation with 
$$ \|W(t)\partial_x^{j+1} \phi\|_{L_x^{2}}=\|\partial_x^{j+1} \phi\|_{L_x^{2}}.$$

Put $u_0(x)=c_0W(-t^*)\phi(x)$ with $0<c_0\ll1$  such that the corresponding solution $u(t)\in C(\mathbb{R};  H^{j+1}(\mathbb{R}))$  is global in time. It is easy to verify $u_0(x)\in H^{j+1}(\mathbb{R})\cap W^{j+1,p}(\mathbb{R})$ with $W(t^*)u_0(x)=\phi(x)\in H^{j+1}(\mathbb{R})$, but $W(t^*)u_0(x)\notin W^{j+1,p}(\mathbb{R})$ for every $p>2$.

Now we prove the second part. It follows from Strichartz estimate (see Lemma \ref{linear estimate:strichartz}) that
\begin{align}
\big\|D_x^{\frac{(p-2)(2j-1)}{4p}}W(t)u_0\big\|_{L^{\frac{4p}{p-2}}_tL^p_{x}}\lesssim \|u_0\|_{L^2_x}.
\label{DoNotProg1}
\end{align}
Taking $\widetilde{\phi}(x)\in H^{j+1}(\mathbb{R})$ and $\widetilde{\phi}(x)\notin W^{r,p}(\mathbb{R}_{+})$ with $r=\frac{(p-2)(2j-1)}{4p}+j+1$, then by \eqref{DoNotProg1} we know that
there exists $t_0>0$ such that
\begin{align}W(\pm t_0)\widetilde{\phi} \in W^{r,p}(\mathbb{R}) \hspace{5mm}\text{and}  \hspace{5mm} W(\pm 2t_0)\widetilde{\phi} \in W^{r,p}(\mathbb{R}).\label{DoNotProg2}
\end{align}

Choosing $u_0=W(t_0)\widetilde{\phi} +W(- t_0)\widetilde{\phi}$ ( multiplying $u_0$ by a small constant, if necessary), then it is easy to see that $u_0\in H^{j+1}(\mathbb{R})\cap W^{r,p}(\mathbb{R})$. The linear part of the global solution with small initial data $u_0$ is
$$W(t)u_0=W(t+t_0)\widetilde{\phi} +W(t- t_0)\widetilde{\phi}$$
which is not in $W^{r,p}(\mathbb{R}_{+})$ at time $t_0$ from \eqref{DoNotProg2} and $\widetilde{\phi}(x)\notin W^{r,p}(\mathbb{R}_{+})$. The same holds for $W(t_0)u_0$.

However, the nonlinear part of the global solution is much more smooth than the linear component by Theorem \ref{ghKdVDuhamelkg2}. Thus, we derive the desired result.  \hspace{37.5mm}$\square$

\section*{Acknowledgments}
M.S is partially supported by the National Natural Science Foundation of China, Grant No.  12101629.

% ----------------------------------------------------------------
%\bibliographystyle{amsplain}
%\bibliography{}

\end{document}